\documentclass[10pt]{scrartcl}
\usepackage{fullpage}
\usepackage{microtype}

\usepackage{amsmath}
\usepackage{amsfonts}
\usepackage{amsthm}
\usepackage{amssymb}

\theoremstyle{plain}
\newtheorem{theorem}{Theorem}[section]
\newtheorem{proposition}{Proposition}[section]
\newtheorem{lemma}{Lemma}[section]

\theoremstyle{definition}
\newtheorem{definition}{Definition}[section]

\usepackage{algorithm}
\usepackage{algpseudocodex}

\usepackage{hyperref}

\usepackage{chemformula}

\usepackage{xfrac}

\usepackage{authblk}
\usepackage[casechanger=auto,style=ext-numeric-comp, sorting=none, giveninits=true, isbn=false, doi=false,date=year,minbibnames=4, maxbibnames=4, articlein=false]{biblatex}
\DeclareFieldFormat[article,inbook,incollection,inproceedings,patent,unpublished]{titlecase:title}{\MakeSentenceCase*{#1}}
\usepackage{siunitx}
\DeclareSIUnit{\cal}{\text{cal}}
\DeclareSIUnit{\hartree}{\textit{E}_h}
\DeclareSIUnit{\angstrom}{\text{\AA}}
\usepackage{mathtools}

\DeclareMathOperator{\conn}{conn}
\DeclareMathOperator{\comp}{comp}

\DeclarePairedDelimiter{\abs}{\lvert}{\rvert}
\DeclarePairedDelimiter{\norm}{\lVert}{\rVert}
\DeclarePairedDelimiter{\pdparen}{\lparen}{\rparen}
\DeclarePairedDelimiter{\pdbrack}{\lbrack}{\rbrack}
\DeclarePairedDelimiter{\pdbrace}{\{}{\}}
\DeclarePairedDelimiter{\pdprod}{\langle}{\rangle}

\providecommand\given{}
\newcommand\SetSymbol[1][]{%
  \nonscript\:#1\vert
  \allowbreak
  \nonscript\:
  \mathopen{}}
\DeclarePairedDelimiterX\Set[1]\{\}{%
  \renewcommand\given{\SetSymbol[\delimsize]}
  #1
}

\newcommand*{\bm}[1]{\mathbf{#1}}
\newcommand{\bbR}{\mathbb{R}}
\newcommand{\bbN}{\mathbb{N}}
\newcommand{\bbZ}{\mathbb{Z}}

\newcommand{\coloneq}{\vcentcolon=}

\begin{document}
\title{On Multilevel Energy-Based Fragmentation Methods}

%\author{James N.\,J. Barker, Michael Griebel, and Jan Hamaekers}
\author[a]{James Barker}
\author[a,b]{Michael Griebel}
\author[b,*]{Jan Hamaekers}

\affil[a]{{\small Institute for Numerical Simulation, University of Bonn, Friedrich-Hirzebruch-Allee 7, 53115 Bonn, Germany}}
\affil[b]{{\small Fraunhofer Institute for Algorithms and Scientific Computing SCAI, Schloss Birlinghoven 1, \mbox{53757 Sankt Augustin}, Germany}}
\affil[*]{{\small Corresponding author: \texttt{jan.hamaekers@scai.fraunhofer.de}}}

\date{\today}

\maketitle

\begin{abstract}
  Energy-based fragmentation methods approximate the potential energy of a molecular system as a sum of contribution terms built from the energies of particular subsystems. Some such methods reduce to truncations of the many-body expansion (MBE); others combine subsystem energies in a manner inspired by the principle of inclusion/exclusion (PIE). The combinatorial technique of Möbius inversion of sums over partially ordered sets, which generalizes the PIE, is known to provide a non-recursive expression for the MBE contribution terms, and has also been connected to related cluster expansion methods. We build from these ideas a very general framework for decomposing potential functions into energetic contribution terms associated with elements of particular partially ordered sets (posets) and direct products thereof. Specific choices immediately reproduce not only the MBE, but also a number of other existing decomposition forms, including, e.g., the multilevel ML-BOSSANOVA schema. Furthermore, a different choice of poset product leads to a setup familiar from the combination technique for high-dimensional approximation, which has a known connection to quantum-chemical composite methods. We present the ML-SUPANOVA decomposition form, which allows the further refinement of the terms of an MBE-like expansion of the Born-Oppenheimer potential according to systematic hierarchies of \emph{ab initio} methods and of basis sets. We outline an adaptive algorithm for the \emph{a posteori} construction of quasi-optimal truncations of this decomposition. Some initial experiments are reported and discussed.
\end{abstract}

\section{Introduction}

The generally high-dimensional nature of the electronic Schrödinger equation renders it inaccessible to conventional numerical techniques. Some recourse is possible to the specialized mechanisms of computational quantum chemistry~\cite{szabo1996, cances2003, echenique2007, helgaker2000, jensen2017}, but the involved costs still scale superlinearly, so these cannot be applied to larger molecular systems, especially not to an acceptable level of accuracy. A common tactic is to decompose the single overarching electronic problem into subproblems. Solutions to these can approximated independently, and combined back into an approximate solution for the original problem. Distinct versions of this tactic are recognisably present in \emph{composite methods}~\cite{zaspel2019, raghavachari2015, karton2022}, and in \emph{energy-based fragmentation methods}~\cite{gordon2011, collins2015, raghavachari2015, herbert2019}.

The working formulae for fragmentation methods are often related to or inspired by the well-known \emph{many-body expansion} (MBE)~\cite{suarez2009, richard2012, chinnamsetty2018, herbert2019}. Extensions apply counting arguments founded in the principle of inclusion/exclusion (PIE)~\cite{richard2012, richard2013}. Significant overlap exists here with those fragmentation methods which construct and identify their fragments using ideas from graph theory~\cite{le2012, weiss2010, heber2014, griebel2014, chinnamsetty2018, ricard2020, zhang2021, seeber2023}, and also with \emph{multilevel fragmentation methods}~\cite{dahlke2007, beran2009, rezac2010, bates2011, chinnamsetty2018, ricard2020, hellmers2021, zhang2021}, which combine the results of subsystem calculations performed at differing levels of theory.

Much work has gone into understanding the properties of and relationships between different fragmentation methods, often with an eye to their unification~\cite{richard2012, richard2013, richard2014, koenig2016, hellmers2021}. We suggest that the combinatorial theory of \emph{Möbius inversion}~\cite{rota1964, aigner1997, stanley2012}, which can be understood as an extension of inclusion/exclusion, provides a useful theoretical framework, not only for the comparison and analysis of existing fragmentation methods, but also for the development of new ones. The connection between Möbius inversion and the MBE is known~\cite{drautz2004, suarez2009}, although apparently unexplored in the modern fragmentation method context. But Möbius inversion has long been applied in cluster expansion methods, see, e.g.,~\cite{klein1986, drautz2004, lafuente2005}. These are closely related to the MBE~\cite{drautz2004}, and some are explicitly founded in graph theory, e.g.,~\cite{klein1986}.

We consider here a very general decomposition form over direct products of partially ordered sets (\emph{posets}). Specific choices of posets directly and deliberately reproduce existing decomposition forms well-known in mathematics. The simplest case would be a totally ordered set, i.e., a chain, and direct products of chains lead to grid-based decompositions like those underlying the \emph{combination technique}~\cite{griebel1992, bungartz2004, gerstner2003, hegland2003, hegland2007, harding2016, wong2016}. The choice instead of a single Boolean algebra, that is, an inclusion-ordered powerset $2^{\pdbrack{n}}$, delivers an \emph{ANOVA-like decomposition}~\cite{griebel2006, kuo2009}, cf.~\cite{rabitz1999}. Any subposet of a Boolean algebra is isomorphic to an ordered subset of the induced subgraphs of some undirected graph $G$ with vertex set $\pdbrack{n}$~\cite{nieminen1980}, and the structural properties of some such subsets make them either more or less appealing for our purposes. From the chemical perspective, the resulting decompositions either mimic or directly reconstruct, e.g., composite methods~\cite{zaspel2019}, the MBE, graph-based forms like the BOSSANOVA (\emph{Bond-Order diSSection ANOVA})~\cite{heber2014, griebel2014} decomposition and particular cases of the CGTCE~\cite{klein1986}, and finally, multilevel mechanisms such as ML-BOSSANOVA~\cite{chinnamsetty2018} and ML-FCR~\cite{hellmers2021}. A product of either a Boolean algebra or a subposet thereof with multiple chains extends these ideas to a new decomposition type that we call \emph{ML-SUPANOVA}.

This article is based on~\cite{barker2024}; further details including data, graphics, and experiments are given there. The structure is as follows: In Section~\ref{sec:decomp_techniques}, we progressively develop and motivate our new multilevel decomposition. Section~\ref{sec:fundamentals} formalizes a hierarchy of model approximations to the Born-Oppenheimer potential function, much like those used in composite methods. Section~\ref{sec:ebfms} establishes some basics regarding fragmentation methods. Section~\ref{sec:inversion} introduces the combinatorial tools that we need to build and operate our multilevel scheme, and demonstrates their application. In Section~\ref{sec:graphs}, we use these tools to uncover an issue in the BOSSANOVA setup, and consider the introduction of convex subgraphs as a partial panacea. Section~\ref{sec:ml} combines ideas from both composite methods and multilevel energy-based fragmentation methods, and introduces the final ML-SUPANOVA decomposition. We outline in Section~\ref{sec:adaptive} an adaptive \emph{a posteori} algorithm for generating cost-effective truncations of this decomposition. Some initial experimental results are discussed in Section~\ref{sec:experiments}, and we make some concluding remarks in Section~\ref{sec:conclusion}.

\section{Decomposition-Based Techniques in Computational Chemistry}\label{sec:decomp_techniques}

\subsection{A model hierarchy and composite methods}\label{sec:fundamentals}

We begin by constructing for later use a formal hierarchy of familiar models in computational quantum chemistry. For general background, we refer to~\cite{szabo1996, cances2003, echenique2007, helgaker2000, jensen2017}. For a molecular system composed of $M$ clamped nuclei and $N$ electrons, the latter with spatial/spin coordinates $\pdbrace{x_i}_{i=1}^N$, we consider the usual electronic Schrödinger equation
\begin{equation}\label{eq:schroedinger}
  H\Psi\pdparen{x_1, \ldots, x_N} = E \Psi\pdparen{x_1, \ldots, x_N}.
\end{equation}
We will refer to as the \emph{Born-Oppenheimer potential} the function
\begin{equation}\label{eq:bo_potential}
  V^{\text{BO}}\pdparen{X_1, \ldots, X_M}
  \coloneq
  \sum_{1 \leq A < B \leq M} \frac{Z_AZ_B}{\norm{R_B - R_A}_2}
  +
  \inf_{\substack{\Psi \in \mathcal{V}\\\pdprod{\Psi, \Psi} = 1}} \ \pdprod{\Psi, H\Psi}
\end{equation}
delivering the ground-state total energy under the Born-Oppenheimer approximation, cf.~\cite{chinnamsetty2018}. This includes the repulsion energy contributed by the nuclei with coordinates $\pdbrace{X_A = (R_A \in \bbR^3, Z_A \in \bbN)}_{A=1}^M$, where each $R_A$ is a spatial coordinate and each $Z_A$ is an atomic number. In~\eqref{eq:bo_potential}, $\mathcal{V}$ is an appropriate solution space for the weak form of~\eqref{eq:schroedinger}, see~\cite{schneider2009, cances2003}.

We restrict our attention to a collection of standard \emph{ab initio} methods which build on the Hartree-Fock (HF) ansatz~\cite{szabo1996} in order to approximate the minimizing $\Psi$ in~\eqref{eq:bo_potential}. These involve a discretization of $H^1\pdparen{\bbR^3}$ via a practically finite \emph{basis set} of $N_{\text{AO}}$ atomic orbital functions, usually Gaussian-type~\cite{gill1994, echenique2007, hill2013, helgaker2000, jensen2017}. The cost scaling behaviour of these methods depends on $N_{\text{AO}}$. For HF, the cost scales formally as $\mathcal{O}\pdparen{N_{\text{AO}}^4}$, and for an additional estimate of the correlation energy from second-order perturbative Møller-Plesset theory (MP2)~\cite{moeller1934, szabo1996}, basically as $\mathcal{O}\pdparen{N_{\text{AO}}^5}$. For better-quality estimates, we turn to coupled-cluster methods: CCSD, CCSDT, etc.~\cite{bartlett2007, schneider2009}. The cost of a general CC$\pdparen{n}$ coupled cluster calculation considering $n$th-order excitations~\cite{kallay2001} can be rather loosely given as $\mathcal{O}\pdparen{N_{\text{AO}}^{2n+2}}$, but see~\cite{kallay2001} for more precision. The perturbative approach underpinning the CCSD(T)~\cite{raghavachari1989} correction has been generalized to CC$\pdparen{n}\pdparen{n+1}$ correction terms applicable to CC$\pdparen{n}$ calculations~\cite{bomble2005, kallay2005, kallay2008}. Again loosely, these cost $\mathcal{O}\pdparen{N_{\text{AO}}^{2n+3}}$ to evaluate, but see and cf\@.~\cite{bartlett2007, kallay2005, karton2022}. Finally, the best possible estimate of the correlation energy for a particular discretization of $H^1\pdparen{\bbR^3}$ is given by full configuration-interaction (FCI), at cost $\sim \mathcal{O}\pdparen{N_{\text{AO}}^N}$, see, e.g.,~\cite{schneider2009}. Taken all together, these methods present a hierarchy containing $2N-1$ elements,\footnote{Where we assume for the sake of simplicity that up to $N$ spin orbitals are available for excitation. In practical settings, we should more properly consider the implications of, e.g., spin-restricted formalisms like RHF, see, e.g.,~\cite{szabo1996, echenique2007}.} with regularly increasing cost and presumably increasing accuracy, cf\@., e.g.,~\cite{bartlett2007, karton2022}:
\begin{equation}\label{eq:qc_method_hierarchy}
  \begin{aligned}
  \text{HF} \pdbrack{\mathcal{O}\pdparen{N_{\text{AO}}^4}}
  &\to
  \text{MP2} \pdbrack{\mathcal{O}\pdparen{N_{\text{AO}}^5}}
  \to
  \text{CCSD} \pdbrack{\mathcal{O}\pdparen{N_{\text{AO}}^6}}
  \to
  \text{CCSD(T)} \pdbrack{\mathcal{O}\pdparen{N_{\text{AO}}^7}}\\
  &\to
  \cdots
  \to
  \text{CC$\pdparen{n}$} \pdbrack{\mathcal{O}\pdparen{N_{\text{AO}}^{2n+2}}}
  \to
  \text{CC$\pdparen{n}\pdparen{n+1}$}\pdbrack{\mathcal{O}\pdparen{N_{\text{AO}}^{2n+3}}}\\
  &\to
  \cdots
  \to
  \text{FCI} \pdbrack{\sim \mathcal{O}\pdparen{N_{\text{AO}}^N}}.
  \end{aligned}
\end{equation}

In practice, the size $N_{\text{AO}}$ of the discretizing basis set is adjustable only indirectly, as a consequence of the generative use of, e.g., one of the cc-pV$n$Z basis sets~\cite{dunning1989}. A different, but also regular quality/cost hierarchy is here exposed as $n$ is varied. The error due to the incompleteness of the basis set is expected to decrease regularly as $n \to \infty$ and the so-called \emph{complete basis set} (CBS) limit is approached~\cite{hill2013}. This regularity is supported at least in practice by the success of extrapolation procedures; see, e.g.,~\cite{feller2011}, and discussion in~\cite{chinnamsetty2018}. In any case, the total number of basis functions required to describe an $M$-atom system with an cc-pV$n$Z basis set goes as $N_{\text{AO}} \sim Mn^3$; see~\cite[Sec.~8.3]{helgaker2000} for precise counting formulae per atom.

These two one-dimensional hierarchies, of method and of basis set, can of course be combined. We will refer to any particular combination of an \emph{ab initio} method, e.g., HF or MP2, and a generating basis set, e.g., cc-pVDZ or cc-pVTZ, as a \emph{level of theory}. Each such pairing is viewed as providing an approximation $V_{m, p} : \pdparen{\bbR^3 \times \bbN}^M \to \bbR$ to the true Born-Oppenheimer potential~\eqref{eq:bo_potential}, where $1 \leq m \leq 2N-1$ indexes the method as per the ordering of~\eqref{eq:qc_method_hierarchy}, and $1 \leq p < \infty$ indexes the basis set within its particular family. In the last, it suffices intuitively to think of the cc-pV$n$Z basis sets, such that $p=1$ indexes cc-pVDZ, $p=2$ similarly cc-pVTZ, and so on, but any naturally-orderable family that produces a sufficiently regular error decay will do.

Similar two-dimensional hierarchies are widely used in \emph{composite methods} like G4(MP2)~\cite{curtiss2007, curtiss2007a}, the ccCA~\cite{deyonker2009}, W4~\cite{karton2006}, and the HEAT schemes~\cite{bomble2006}; see also the reviews~\cite{raghavachari2015, karton2022}. As a general rule, these require some set of full-system single-point calculations to be carried out, each at a different level of theory. The energetic results of some of these are carefully subtracted from those of others, providing a collection of correction terms that represent, e.g., the added accuracy gained by moving from MP2/cc-pVDZ to CCSD/cc-pVDZ, or from MP2/cc-pVDZ to MP2/cc-pVTZ. The sum of these correction terms then gives an overall extrapolation towards a true FCI/CBS solution.

It was observed in passing in~\cite{chinnamsetty2018} that the derivation and resummation of the correction terms in composite methods is formally reminiscent of the \emph{combination technique}~\cite{griebel1992, bungartz2004} for efficient high-dimensional approximation. In~\cite{zaspel2019}, this similarity was exploited to construct a \emph{combination quantum machine-learning} scheme. For intuition, it helps to interpret the indices $m$ and $p$ as locating the potentials $V_{m,p}$ on a so-called Pople diagram; see and cf\@. the reviews already cited,~\cite[Fig.~5.4]{jensen2017}, and~\cite[Fig.~1]{zaspel2019}. Without involving machine learning, a generalized picture of the connection can be given as follows, cf.~\cite{zaspel2019, karton2022}, although at some risk of oversimplification.\footnote{For example, this does not allow for some common characteristics of composite models, such as the prevalent use of CBS extrapolation schemes, and the introduction of corrections targeting, e.g., core-valence interactions and relativistic effects. For a much more detailed discussion of this setup, see~\cite[Chap.~4]{barker2024}.} Define a family of contribution/correction terms by
\begin{equation*}%\label{eq:composite_contributions}
  \tilde{V}_{m, p}
  \coloneq V_{m,p} - \sum_{n < m} \tilde{V}_{n, p} - \sum_{q < p} \tilde{V}_{m, q} - \sum_{\substack{n < m\\q < p}} \tilde{V}_{n, q}.
\end{equation*}
Let $\mathcal{L}$ be a linear functional mapping potentials $V_{m,p}$ into some suitable Banach space $Y$, and assume much as in~\cite{chinnamsetty2018} that $\norm{\mathcal{L}\pdbrack{V_{m, p} - V_{m, p-1}}}_Y \lesssim g_p$, for some sequence $\pdparen{g_p \in \bbR}_{p \in \bbN}$ such that $\sum_{p \in \bbN} g_p$ converges absolutely. We restrict ourselves here to a simple point-evaluation functional $\mathcal{L}\pdbrack{V_{m,p}} = V_{m,p}\pdparen{X_1, \ldots, X_M}$ parametrized by a fixed molecular conformation $\pdbrace{X_A}_{A=1}^M$, but, e.g., an evaluation of the nuclear gradient requires no significant theoretical adjustment in what follows. Then we obtain pointwise an exact expansion of the Born-Oppenheimer potential as
\begin{equation*}\label{eq:bo_composite_expansion}
  V^{\text{BO}} = \sum_{m=1}^{2N-1} \sum_{p=1}^{\infty} \tilde{V}_{m,p}.
\end{equation*}
Experience with the combination technique suggests that accurate yet affordable results can be obtained by truncating this sum in a lower-triangular fashion, for example, after all terms with $m + p \leq N'$ for some reasonably low $N'$. And indeed, up to very many intricacies, and precise choice of the set of potentials $V_{m,p}$ and truncation of~\eqref{eq:bo_composite_expansion}, this is essentially just how composite methods operate.

\subsection{Energy-based fragmentation methods}\label{sec:ebfms}

Alternatively, \emph{energy-based fragmentation methods}, hereafter just ``fragmentation methods'', aim to obtain approximate solutions to the Schrödinger equation at reduced computational costs compared to the standard full-system methods such as those outlined above. For detailed reviews, see, e.g.,~\cite{gordon2011, collins2015, raghavachari2015, herbert2019}. In brief, a fragmentation method involves a decomposition of the full set of nuclear indices $\pdbrack{M} = \Set{1, 2, \ldots, M}$ into a family of $K$ subsets $F = \pdbrace{F_i \subseteq \pdbrack{M}}_{i=1}^K$ such that $\bigcup_{i=1}^K F_i = \pdbrack{M}$. As per, e.g.,~\cite{suarez2009, richard2012, raghavachari2015}, this decomposition can be either \emph{disjoint}, so that $F_i \cap F_j = \emptyset$ for $i \neq j$ and thus $\pdbrace{F_i}_{i=1}^K$ is a strict partition of $\pdbrack{M}$, or \emph{overlapping} if not; cf\@., however, comments in~\cite{collins2015} questioning the meaningfulness of this differentiation in practice. We will refer to each $F_i$ as a \emph{fragment}, and the family $F = \pdbrace{F_i \subseteq \pdbrack{M}}_{i=1}^K$ as a \emph{fragmentation}, with the disjointness condition $F_i \cap F_j = \emptyset$ holding unless otherwise stated.

The prototypical fragmentation method involves a truncation of the well-known \emph{many-body expansion} (MBE), see, e.g.,~\cite{drautz2004, richard2012, chinnamsetty2018}. We construct the MBE as follows. Write simply as $V : \pdparen{\bbR^3 \times \bbN}^M \to \bbR$ some particular symmetric potential, either an approximation $V_{m, p}$ as defined above, or the true $V^{\text{BO}}$. Further, let $\pdbrace{V_{\bm{u}} : \pdparen{\bbR^3 \times \bbZ}^M \to \bbR}_{\bm{u} \subseteq \pdbrack{M}}$ be a family of \emph{subproblem potentials}, one for each subset of the nuclear indices $\pdbrack{M}$. Each $V_{\bm{u}}$ is a potential which focuses basically the same level of theory as $V$ on the nuclei indexed by $\bm{u}$; however, each such potential still retains at least the possibility of a dependence on the full set of nuclei. This allows, for instance, the definition of each $V_{\bm{u}}$ to cater for the introduction of such link atoms as are necessary after severing the subsystem indexed by $\bm{u}$ out from its surrounding electrostatic environment~\cite{field1990}, and/or to embed a single-point calculation in a field of Coulomb point charges, as in, e.g.,~\cite{dahlke2007, liu2016}. The MBE of $V$ is then
\begin{equation}\label{eq:nuclear_mbe}
  V
  =
  \sum_{\bm{u}\subseteq\pdbrack{M}} \tilde{V}_{\bm{u}}
  = \tilde{V}_{\emptyset}
  + \sum_{i=1}^M \tilde{V}_{\Set{i}}
  + \sum_{i=1}^M\sum_{j=i+1}^M \tilde{V}_{\Set{i,j}} + \cdots + \tilde{V}_{\pdbrack{M}},
\end{equation}
where each member of the family $\pdbrace{\tilde{V}_{\bm{u}}}_{\bm{u} \subseteq \pdbrack{M}}$ is a \emph{contribution potential}, defined for the moment recursively as
\begin{equation}\label{eq:mbe_contribution_recursive}
  \tilde{V}_{\bm{u}} \coloneq V_{\bm{u}} - \sum_{\bm{v} \subset \bm{u}} \tilde{V}_{\bm{v}}.
\end{equation}
As noted in~\cite{chinnamsetty2018}, it is easy to see that as long as $V_{\pdbrack{M}} = V$, the expansion~\eqref{eq:nuclear_mbe} is exact, regardless of the form of the remaining potentials $V_{\bm{u}}$. Although $V_{\emptyset}$ can be zero, and usually implicitly is, this is not a requirement, as we shall briefly discuss below.

It is more common in practice to construct an MBE in terms of some fragmentation $F$ of $\pdbrack{M}$, as in, e.g.,~\cite{koenig2016}. To formulate this, writing $F_{\bm{u}} \coloneq \bigcup_{i \in \bm{u}} F_i$, we decompose $V$ instead as
\begin{equation}\label{eq:fragment_mbe}
  V = \sum_{\bm{u}\subseteq\pdbrack{K}} \tilde{V}_{F_{\bm{u}}},
  \qquad
  \tilde{V}_{F_{\bm{u}}} \coloneq V_{F_{\bm{u}}} - \sum_{\bm{v} \subset \bm{u}} \tilde{V}_{F_{\bm{v}}}.
\end{equation}
We will refer to this form as a \emph{fragment} MBE, in contrast to the \emph{nuclear} MBE~\eqref{eq:nuclear_mbe}. For notational ease, we will work mostly with the nuclear form. But as we shall see, a fragment MBE can be viewed quite precisely as a restriction of a nuclear MBE.

The traditional way to derive a working fragmentation method from an MBE like~\eqref{eq:nuclear_mbe} (or~\eqref{eq:fragment_mbe}, proceeding equivalently) is to truncate the expansion after all terms $\tilde{V}_{\bm{u}}$ with $\abs{\bm{u}} \leq n \leq M$ for some particular choice of $n$. The critical assumption driving such an \emph{$n$-body expansion} is that of a decay in $\abs{\tilde{V}_{\bm{u}}}$ as $\abs{\bm{u}}$ increases, one swift enough that $\sum_{\abs{\bm{u}} > n} \tilde{V}_{\bm{u}}$ is negligibly small even for $n \ll M$~\cite{chinnamsetty2018, herbert2019}. Then, since the evaluation costs of the potentials scale polynomially in $\abs{\bm{u}}$, the terms in the $n$-body expansion are individually cheap to calculate, and this can collectively be done in an embarrassingly-parallel fashion~\cite{herbert2019}. It has been demonstrated in~\cite{richard2014, lao2016} that these assumptions are flawed. The actual decay in the potentials is not as reliable or as fast as hoped, see also, e.g.,~\cite{ouyang2014}, but more concerningly, the arithmetic required to evaluate an $n$-body expansion rapidly amplifies the uncertainties produced by the iterative solvers which always lie behind the potentials $V_{\bm{u}}$ in practical implementation. We will briefly return to this latter point in Sections~\ref{sec:adaptive} and~\ref{sec:experiments} below.

Other fragmentation methods start with an explicitly overlapping set of fragments, and generate weighted sums of energy terms motivated by inclusion/exclusion arguments in an effort to somehow avoid double-counting interactions or particles; see~\cite{richard2012, richard2013}, and cf., e.g.,~\cite{weiss2010}. The \emph{generalized many-body expansion} (GMBE)~\cite{richard2012} seeks to build by extension a framework capable of handling $n$-body combinations of overlapping fragments, rather than disjoint fragments as in~\eqref{eq:fragment_mbe}. To obtain an $n$-body truncation of the GMBE, one starts by forming the family $\pdbrace{F'_i}_{i=1}^{K'}$ of all $K' = \binom{K}{n}$ exactly $n$-fold unions of a set of $K$ distinct and potentially overlapping fragments $F_i$, called in context \emph{$n$-mers}~\cite{richard2012, richard2013}. Adapting from~\cite[(1.9)]{richard2013}, an $n$-body truncation of the GMBE can then be written as
\begin{equation}\label{eq:gmbe}
  E^{\text{GMBE}}_{\pdparen{n}}
  =
  \sum_{i=1}^{K'} E_{F'_i}
  - \sum_{i < j}^{K'} E_{F'_i \cap F'_j}\\
  + \sum_{i < j < k}^{K'} E_{F'_i \cap F'_j \cap F'_k}
  + \cdots
  + \pdparen{-1}^{K'+1}E_{F'_1 \cap \cdots \cap F'_{K'}},
\end{equation}
where, e.g., $E_{F'_i \cap F'_j} \coloneq \mathcal{L}\pdbrack{V_{F'_i \cap F'_j}}$ indicates the total energy of the subsystem indexed by the intersection of $n$-mer fragments $F'_i$ and $F'_j$. The formula is explicitly motivated by the well-known inclusion/exclusion expression for the cardinality of the union of a collection of potentially-overlapping sets. In the case when the initial fragmentation is disjoint, the $n$-body GMBE energy is expected to reduce to a standard $n$-body truncation of a fragment MBE~\cite{richard2012, richard2013}.

Mathematically, the MBE is an \emph{ANOVA-like decomposition}~\cite{griebel2006, kuo2009, heber2014, griebel2014, chinnamsetty2018}, and so each term $\tilde{V}_{\bm{u}}$ can be viewed as the contribution of the lower-dimensional subset of dimensions indexed by $\bm{u}$ to the high-dimensional function $V$. Just the same idea underlies a construction of the MBE as a particular \emph{high-dimensional model representation} (HDMR)~\cite{rabitz1999}. A low-order truncation of the MBE therefore exploits a low effective dimension in the underlying electronic problem~\cite{griebel2006, chinnamsetty2018}. It has been widely observed, see, e.g.,~\cite{suarez2009, collins2015, raghavachari2015, koenig2016}, that the energy expressions of many fragmentation methods can be understood as truncated MBEs. We believe, probably uncontroversially, that the ANOVA-like viewpoint suggests that \emph{any} fragmentation method should ultimately revert to a truncation $\sum_{\bm{u} \in I} \tilde{V}_{\bm{u}}$ for a downward-closed subset $I \subseteq 2^{\pdbrack{M}}$, in the sense that if $\bm{v} \subseteq \bm{u} \in I$ then also $\bm{v} \in I$ --- for if not, the ``missing'' terms in $I$ introduce an inherent error to the approximation that cannot be systematically removed. This requirement has already been explicitly stated in some methods, e.g.,~\cite{koenig2016, chinnamsetty2018}.

The following straightforward observation underpins much of what we do in the following: any truncation of~\eqref{eq:nuclear_mbe} can be rewritten as a linear combination of potentials $V_{\bm{u}}$ rather than explicitly contributions $\tilde{V}_{\bm{u}}$, simply by application of~\eqref{eq:mbe_contribution_recursive} and collection of terms. Doing so leads to non-recursive expressions for, e.g., standard $n$-body expansions~\cite{kongsted2006, richard2014}. We highlight one particular existing setup which already acknowledges and exploits this for more arbitrary truncations. This \emph{fragment combination range} (FCR) method~\cite{koenig2016, hellmers2021} is derived from a fragment MBE like~\eqref{eq:fragment_mbe}. Adapting notation from~\cite{koenig2016} with reference also to~\cite[(9), (10), (11)]{hellmers2021}, the total energy of a molecular system composed of $K$ disjoint subsystems represented by composite coordinates $\pdbrace{z_i}_{i=1}^K$ is approximated  in the FCR approach as
\begin{equation}\label{eq:fcr}
  E\pdparen{z_1, \ldots, z_K}
  \approx
  \sum_{\bm{f}_l \in \pdbrace{\text{FCR}}} p^{\text{FCR}}_{\bm{f}_l} E_{\bm{f}_l}\pdparen{\pdbrace{z}_{\bm{f}_l}}.
\end{equation}
Here, $\pdbrace{\text{FCR}}$ indicates an arbitrary downward-closed subset of $2^{\pdbrack{K}}$, and $E_{\bm{f}_l}\pdparen{\pdbrace{z}_{\bm{f}_l}}$ the total energy of the combined subsystem specified by those composite coordinates $\pdbrace{z_i}_{i \in \bm{f}_l}$. The coefficients $p^{\text{FCR}}_{\bm{f}_l}$ are given non-recursively by
\begin{equation}\label{eq:fcr_coefficient}
  p^{\text{FCR}}_{\bm{f}_l}
  = \sum_{\substack{\bm{f}_{l'} \supseteq \bm{f}_l\\\bm{f}_{l'} \in \pdbrace{\text{FCR}}}}
  \pdparen{-1}^{l' - l},
\end{equation}
derived in~\cite{koenig2016} by an inductive argument. It is noted that in some cases, these coefficients vanish.

\subsection{Fragmentation methods from the perspective of Möbius inversion}\label{sec:inversion}

The principle of inclusion/exclusion, widely used to derive fragmentation method energy formulae like~\eqref{eq:gmbe}, can be substantially generalized by the technique of \emph{Möbius inversion} from combinatorics and the theory of partially ordered sets~\cite{rota1964, aigner1997, stanley2012}. For the convenience of the reader, we will reintroduce some basic ideas, following particularly~\cite{stanley2012} with minor notational deviation and general reference also to~\cite{rota1964, aigner1997}. Core concepts are illustrated in Figures~\ref{fig:boolean_algebra},~\ref{fig:convex_subgraphs_subposet}, and~\ref{fig:ml_supanova_example_grid}, and discussed in sections to come.

Recall, informally, that a partially ordered set, or \emph{poset}, is a set $P$ with an abstract ordering so that $s \leq t$ for some particular pairs $s, t \in P$. The idea of ``$\leq$'' is sufficient to define also ideas like, e.g., $s < t$ and $s \geq t$. A poset where always either $s \leq t$ or $t \leq s$ for any $s, t \in P$ is called a \emph{chain}. A poset $P$ is \emph{locally finite} if $\Set{ u \in P \given s \leq u \leq t}$ is finite for every $s \leq t$.

\begin{definition}[Möbius function]
  The \emph{Möbius function} of a locally finite poset $P$ is
  \begin{equation}\label{eq:moebius_function}
    \mu_P\pdparen{s, t} = \begin{cases}
      1&\text{if $s = t$,}\\
      -\sum_{s \leq u < t} \mu\pdparen{s, u}&\text{if $s < t$,}\\
      0&\text{otherwise.}
    \end{cases}
  \end{equation}
\end{definition}

This is essentially the definition in~\cite{stanley2012}, but extended to explicitly include the case when $s \not\leq t$. For particular posets $P$, simpler expressions exist for $\mu_P$. For example, suppose that $P$ is a chain poset, and that it has a $\hat{0}$, i.e., a unique minimal element. In this case, it is immediate from~\eqref{eq:moebius_function} that
\begin{equation}\label{eq:chain_poset_moebius}
  \mu_P(s, t) = \begin{cases}
    1&\text{if $s = t$,}\\
    -1&\text{if $s \prec t$},\\
    0&\text{otherwise,}
  \end{cases}
\end{equation}
where $s \prec t$ indicates that $s$ is \emph{covered} by $t$, i.e., $s < t$ and there exists no intervening $u \in P$ with $s < u < t$. Also, $P$ is isomorphic to $\bbN$, in the sense that there exists a bijection $\phi : P \to \bbN$ which preserves $\leq$, as does its inverse $\phi^{-1}$.\footnote{Informally, just label $\hat{0}$ as 0, then its cover $s \succ \hat{0}$ as 1, and so on.} Since the Möbius function is invariant under isomorphism, see also~\cite[Cors.~2 and 3]{bender1975}, the form of $\mu_\bbN$ follows immediately. Further, write $B_n$ to be the powerset $2^{\pdbrack{n}}$ ordered such that $\bm{u} \leq \bm{v}$ exactly when $\bm{u} \subseteq \bm{v}$. This poset is called the \emph{Boolean algebra} of rank $n$, and~\eqref{eq:chain_poset_moebius} can be used to show that
\begin{equation}\label{eq:boolean_algebra_moebius_fn}
  \mu_{B_n}(\bm{u}, \bm{v}) = \begin{cases}
    \pdparen{-1}^{\abs{\bm{v} - \bm{u}}}&\text{if $\bm{u} \leq \bm{v}$,}\\
    0&\text{otherwise.}
  \end{cases}
\end{equation}

The next result is a slightly specialized form of~\cite[Prop.~3.7.1]{stanley2012}, given without proof; historically, see~\cite{rota1964}. We recall that the \emph{principal order ideal} of some  $t \in P$ is the set $\Lambda_t = \Set{s \in P \given s \leq t}$.

\begin{theorem}[Möbius inversion~\cite{rota1964, stanley2012}]\label{thm:moebius_inversion}
  If $P$ is a poset such that $\Lambda_t$ is finite for every $t \in P$, and $f$, $g : P \to \mathbb{R}$, then
  \begin{equation*}
    g(t) = \sum_{s \leq t} f(s)
  \end{equation*}
  for all $t \in P$ if and only if
  \begin{equation*}
    f(t) = \sum_{s \leq t} \mu_P(s, t) g(s)
  \end{equation*}
  for all $t \in P$.
\end{theorem}

A non-recursive expression for the terms $\tilde{V}_{\bm{u}}$ in an MBE like~\eqref{eq:nuclear_mbe} can be obtained pointwise via Möbius inversion~\cite{drautz2004, suarez2009}. To make the bidirectional nature of the process clear, note that we can equally well just define outright each
\begin{equation}\label{eq:mbe_contribution_nonrecursive}
  \tilde{V}_{\bm{u}}
  \coloneq
  \sum_{\bm{v} \subseteq \bm{u}} \pdparen{-1}^{\abs{\bm{v}-\bm{u}}} V_{\bm{v}}
  =
  \sum_{\bm{v} \subseteq  \bm{u}} \mu_{B_M}\pdparen{\bm{v}, \bm{u}} V_{\bm{v}},
\end{equation}
and then obtain~\eqref{eq:mbe_contribution_recursive} as an immediate consequence of Theorem~\ref{thm:moebius_inversion}. The equivalence between~\eqref{eq:mbe_contribution_recursive} and~\eqref{eq:mbe_contribution_nonrecursive} is known even without explicit recourse to Möbius inversion~\cite{koenig2016}, and has also been noted in the ANOVA setting; cf\@.~\cite[Thm.~2.1]{kuo2009}. Interestingly, a simpler variant of Theorem~\ref{thm:moebius_inversion} is invoked for this purpose in~\cite{weiss2010}, where it is named as an extension of the PIE; cf\@. again~\cite{stanley2012}.

\begin{figure}[t]
  {\centering
    \includegraphics{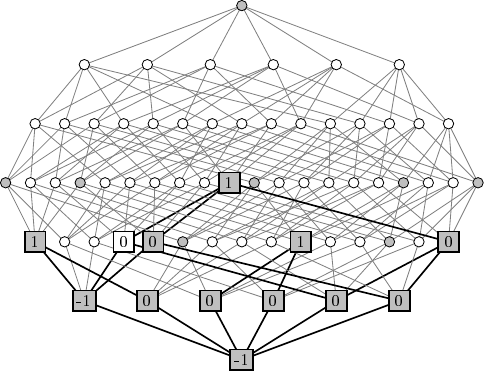}

  }
  \caption{A Hasse-like diagram of the Boolean algebra $B_6$. Each vertex represents a subset $\bm{u} \subseteq \pdbrack{6}$, and an upwards line between vertices $\bm{u}$ and $\bm{v}$ indicates that $\bm{u} \prec \bm{v}$. The bottom-most point is the empty set, $\emptyset$, those on the first row, left to right, $\pdbrace{1}, \pdbrace{2} \ldots$, those on the second, $\pdbrace{1, 2}, \pdbrace{1, 3}, \ldots$, and so on. Each point thus corresponds to a term $\tilde{V}_{\bm{u}}$ of the nuclear MBE~\eqref{eq:nuclear_mbe} of a six-atom system, or alternatively a term $\tilde{V}_{F_{\bm{u}}}$ of the fragment MBE~\eqref{eq:fragment_mbe} for a system of six disjoint fragments. Rectangular vertices indicate elements of a particular order ideal $I$ of this poset, or equivalently, the terms of a downwards-closed $I$-truncation of the MBE. The numbers within are their combination coefficients, as per~\eqref{eq:combination_coefficient}; so if the truncation were rewritten as a sum of contribution potentials $V_{\bm{u}}$ (or their equivalents in the fragment case), the term $V_{\emptyset}$ would be weighted by $-1$, the term $V_{\pdbrace{3,4}}$ by $1$, and the term $V_{\pdbrace{5}}$ would not appear. Vertices shaded gray have meaning in the context of Figure~\ref{fig:convex_subgraphs_subposet}.}
  \label{fig:boolean_algebra}
\end{figure}

Posets are usually represented visually as \emph{Hasse diagrams}. An example Hasse diagram for the Boolean algebra $B_6$ of subsets $\bm{u} \subseteq \pdbrack{6}$ is shown in Figure~\ref{fig:boolean_algebra}, annotated to provide a concrete example of some of the notation and ideas that we will now introduce.\footnote{A connection between a specific kind of Hasse diagram and the terms of the GMBE is made in~\cite{broderick2023}; we will return briefly to this point in Section~\ref{sec:adaptive} below.} These latter are given in a very general form. They can be immediately specialized to the case of $B_M$, and therefore the MBE, but they will also apply to the general expansion form~\eqref{eq:poset_grid_multilevel} to which we will progressively build.

In combinatorial terminology, an \emph{order ideal} is a subset $I$ of a poset $P$ such that if $t \in I$, then also $s \in I$ for all $s \leq t$. Obviously, a downward-closed subset of $B_M$ is just such an order ideal.

\begin{definition}[Combination coefficients]
  Let $P$ be a locally finite poset with a $\hat{0}$, and let $I$ be a finite order ideal of $P$.  The \emph{combination coefficient} of each $s \in P$ for $I$ is defined to be
  \begin{equation}\label{eq:combination_coefficient}
    D^{\pdparen{I}}_s = \sum_{\substack{t\in I\\t \geq s}} \mu_P\pdparen{s, t}.
  \end{equation}
\end{definition}

The name ``combination coefficient'' stems from the combination technique, and there are direct connections in what follows to poset-based constructions in that area; see, e.g.,~\cite{hegland2003, hegland2007, harding2016, wong2016} and the more general development in~\cite[Chap.~3]{barker2024}. From our perspective, the combination coefficients arise by rewriting what we will call an \emph{$I$-truncation} of the MBE as a linear combination of the subproblem potentials $V_{\bm{u}}$, i.e.,
\begin{equation}\label{eq:mbe_combination_sum}
  \begin{split}
  S_I \coloneq \sum_{\bm{u} \in I} \tilde{V}_{\bm{u}}
  &= \sum_{\bm{u} \in I} \sum_{\bm{v} \subseteq \bm{u}} V_{\bm{v}}\mu_{B_M}\pdparen{\bm{v}, \bm{u}} \\
  &= \sum_{\bm{u} \in I} V_{\bm{u}}
  \sum_{\substack{\bm{v} \supseteq \bm{u}\\\bm{v} \in I}} \mu_{B_M}\pdparen{\bm{v}, \bm{u}}
  = \sum_{\bm{u} \in I} D^{\pdparen{I}}_{\bm{u}} V_{\bm{u}},
  \end{split}
\end{equation}
by~\eqref{eq:mbe_contribution_nonrecursive} and rearrangement; cf\@. the alternative proof of~\cite[Prop.~3.7.1]{stanley2012}. Note that when $\mu_{B_M}\pdparen{\bm{v}, \bm{u}} = \pdparen{-1}^{\abs{\bm{v}-\bm{u}}}$ is inserted in~\eqref{eq:combination_coefficient}, then $D^{\pdparen{I}}_{\bm{u}}$ is exactly the FCR coefficient $p^{\text{FCR}}_{\bm{f}_l}$ in~\eqref{eq:fcr_coefficient}, up to notation. So our setup can be viewed as just rederiving the FCR energy expression~\eqref{eq:fcr} by application of heavy machinery.

Now suppose that $\hat{F} \subseteq B_M$ is an arbitrary subposet of $B_M$, recalling in general that a \emph{subposet} of some poset $P$ is a subset $Q \subseteq P$ equipped with an appropriate restriction of the same order relation $\leq$ as $P$. Then we can decompose
\begin{equation}\label{eq:supanova}
    V \approx \sum_{\bm{u} \in \hat{F}} \tilde{V}'_{\bm{u}},
\end{equation}
in terms of contribution potentials defined only for elements of $\hat{F}$,
\begin{equation*}
\tilde{V}'_{\bm{u}}
  \coloneq
  V_{\bm{u}} - \sum_{\bm{v} <_{\hat{F}} \bm{u}} \tilde{V}'_{\bm{v}}
  =
  \sum_{\bm{v} \leq_{\hat{F}} \bm{u}} \mu_F\pdparen{\bm{v}, \bm{u}} V_{\bm{v}},
\end{equation*}
where the last equality follows again by Möbius inversion. For reasons that will become clear below, we will refer to~\eqref{eq:supanova} as a SUPANOVA expansion, but it can also be viewed as a special case of the CGTCE decomposition form given in~\cite{klein1986}. If $\pdbrack{M} \in \hat{F}$ and $V_{\pdbrack{M}} = V$, then the decomposition~\eqref{eq:supanova} is exact, but even if not, a sum $S'_{I'} = \sum_{\bm{u} \in I'} \tilde{V}'_{\bm{u}}$ can be defined for any order ideal $I'$ of $\hat{F}$, and rewritten as a linear combination of only those original contribution potentials $V_{\bm{u}}$ where $\bm{u} \in \hat{F}$. Such a truncation $S'_{I'}$ can be exactly identified with some $S_I$ for an order ideal of the full $B_M$ at least when the non-zero terms of $S_I$ are exactly those of $S'_{I'}$. We can formalize this in the general setting as follows:

\begin{definition}[Combination-consistency]
  Let $P$ be a locally finite poset, and $Q$ be a subposet of $P$. Further, let $I$ be a finite order ideal of $P$, and $I'$ be a finite order ideal of $Q$. Write as above $D^{\pdparen{I}}_s$ to be the combination coefficient of $s \in P$ for $I$, and for notational clarity, write $\hat{D}^{\pdparen{I'}}_{s'}$ to be the combination coefficient of $s' \in Q$ for $I'$. If for every $s \in P$ it holds that
  \begin{equation}\label{eq:combination_consistency}
    D^{\pdparen{I}}_s = \begin{cases}
      \hat{D}^{\pdparen{I'}}_s & \text{if $s \in Q$,}\\
      0 & \text{otherwise,}
    \end{cases}
  \end{equation}
  then $I$ and $I'$ are \emph{combination-consistent}. If a combination-consistent finite order ideal $I$ of $P$ exists for every finite order ideal $I'$ of $Q$, then $Q$ is a \emph{combination-consistent} subposet of $P$.
\end{definition}

\begin{figure}[t]
  {\centering
    \includegraphics{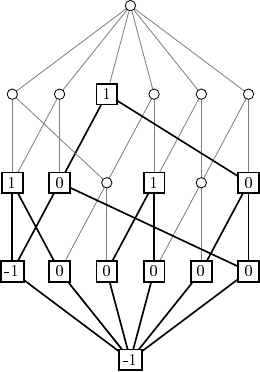}

  }
  \caption{A Hasse-like diagram of the subposet of $B_6$ formed by those subsets shaded gray in Figure~\ref{fig:boolean_algebra}. Again, rectangular vertices show the combination coefficients of an order ideal $I'$ of this subposet, specifically that generated (in this subposet) by the maximal elements of the order ideal $I$ shown in Figure~\ref{fig:boolean_algebra}. These are identical to the coefficients for the respective elements in $I$ shown there, and conversely, the coefficient of the single element of $I$ not in this subposet is zero. So $I'$ is combination-consistent with $I$. This particular subposet has a relevance to, e.g., the benzene molecule, which will become clear in Section~\ref{sec:graphs} below.}
  \label{fig:convex_subgraphs_subposet}
\end{figure}

An example of a particular subposet of $B_6$ is drawn in Figure~\ref{fig:convex_subgraphs_subposet}. The order ideal $I'$ of the subposet shown in this diagram is combination-consistent with the order ideal $I$ of $B_6$ shown in Figure~\ref{fig:boolean_algebra}; this can be verified by comparison of the combination coefficients shown in the two figures. In fact, for reasons that will become clear below, the entire subposet is combination-consistent with $B_6$, although this is not immediately obvious.

If we assume that a general poset $P$ possesses a certain kind of structure, we can characterize the necessary requirements for some subposet $Q$ of $P$ to be combination-consistent with it. This characterization slightly generalizes work in~\cite{lafuente2005}. For a detailed explanation of the connection, see~\cite[Sec.~5.2]{barker2024}. But in short, a cluster expansion is there truncated after a certain order ideal of a direct product of subposets of a Boolean algebra, and Möbius inversion leads to a different expression for what we call combination coefficients. This expression can be recreated in our setup as follows, as an immediate consequence of~\eqref{eq:combination_coefficient} and~\eqref{eq:moebius_function}. A direct connection also exists here to a textbook derivation of the set-cardinality PIE expression via Möbius inversion; see and cf.~\cite[p.~265]{stanley2012}. The PIE in turn connects back to the combination technique; see, e.g.,~\cite{hegland2007, wong2016}.

\begin{lemma}\label{lem:equivalent_combination_coefficient}
Let $I$ be an order ideal of a locally finite poset $P$. Define $J \coloneq I \cup \pdbrace{\hat{1}_J}$ ordered just as $I$ but with additionally $s < \hat{1}_J$ for every $s \in I$. The precise choice of $\hat{1}_J$ is not important here, only that there exists some unique maximal element of $J$. Then, for every $s \in I$,
\begin{equation*}
    D^{\pdparen{I}}_s = -\mu_{J}\pdparen{s, \hat{1}_J}.
\end{equation*}
\end{lemma}

We need some further basic definitions; see again and cf., e.g.,~\cite{aigner1997, stanley2012}. A \emph{meet semilattice} is a poset $P$ where every pair of elements $s$, $t \in P$ have a unique greatest lower bound $r$ in $P$;
this bound is called their \emph{meet} and written as $r = s \wedge t$. If $Q$ is a meet-closed subposet of a meet semilattice $P$, i.e., one where $s \wedge_P t \in Q$ for any $s$, $t \in Q$, then we call $Q$ a \emph{meet subsemilattice} of $P$. For a finite order ideal $I$ of a general poset $P$, the \emph{generating antichain} of $I$ is the set $A = \Set{s \in I \given \nexists\,t \in I,\,t > s}$ of maximal elements of $I$; we will write then $I = \pdprod{A} \coloneq \bigcup_{a \in A} \Lambda_a$. Using this terminology, we can rephrase a key observation made on~\cite[p.~7481]{lafuente2005} in a more general setting. The proof is given in Appendix~\ref{sec:proofs}, and uses just the same ideas as in~\cite{lafuente2005}.

\begin{proposition}[Generalized from~\cite{lafuente2005}]\label{prop:meets_of_generating_antichain}
  Let $P$ be a locally finite meet semilattice, and $I$ be a finite order ideal of $P$ with generating antichain $A$. If $D^{\pdparen{I}}_s \neq 0$ for some $s \in P$, then there exists some $\Set{a_1, \ldots, a_n} \subseteq A$ such that $s = a_1 \wedge \cdots \wedge a_n$.
\end{proposition}

Using our notation and terminology, and with a little generalization, it was further observed in~\cite{lafuente2005} that the subposet $Q$ built from the meets of an essentially arbitrary antichain $A$ of a particular meet semilattice $P$ is combination-consistent with $I = \pdprod{A}_P \subseteq P$. It is straightforward to check that, given an order ideal $I'$ of some arbitrary subposet $Q$ of a more general $P$,  the only possible combination-consistent order ideal $I$ of $P$ is that provided by the generating antichain of $I'$, if one exists at all. For the proof, see Appendix~\ref{sec:proofs}.

\begin{lemma}[Adapted 
from Lem.~5.2.7~\cite{barker2024}]\label{lem:unique_consistent_order_ideal}
  Let $P$ be a locally finite poset, let $Q$ be a subposet of $P$, and let $I' \subseteq Q$ be an arbitrary finite order ideal of $Q$. If $I$ is a finite order ideal of $P$ which is combination-consistent with $I'$, then $I = \pdprod{A'}_P$, where $A' \subseteq I'$ is the generating antichain of $I'$.
\end{lemma}

Since the next result is an immediate consequence of standard results in order theory, it is hardly interesting in its own right, although we do not believe we have seen it explicitly stated as such. But it is very important for us. The key ideas, particularly for direction $\pdparen{\Leftarrow}$, are again to be found in~\cite{lafuente2005} for a particular lattice, and their generalization to an arbitrary meet semilattice is immediate. The proof is relegated to Appendix~\ref{sec:proofs}. We remark that other lattices, related to $\bbN^d$, also arise very naturally in the standard grid-based setting of the combination technique~\cite{hegland2003, hegland2007, harding2016, wong2016}, and there is also an especially close connection here to an inclusion/exclusion-based construction in~\cite[Chap.~3]{wong2016}; see~\cite[Sec.~5.2.3]{barker2024} for deeper discussion.

\begin{theorem}[Thm.~5.2.8~\cite{barker2024}]\label{thm:combination_consistency}
  Let $P$ be a locally finite meet semilattice with a $\hat{0}$, and let $Q$ be a subposet of $P$. Then $Q$ is combination-consistent with $P$ if and only if $Q$ is a meet subsemilattice of $P$.
\end{theorem}

As a first consequence of this, consider some fragmentation $F =\pdbrace{F_i \subseteq \pdbrack{M}}_{i=1}^K$ of $\pdbrack{M}$, and let in this case $\hat{F} = \Set{F_{\bm{u}} \given \bm{u} \in B_K}$. Since meets in $B_M$ are given by set intersections, it is not hard to see that $\hat{F}$ is a meet subsemilattice of $B_M$. This means that any downward-closed truncation of a fragment MBE~\eqref{eq:fragment_mbe}, standard $n$-body or otherwise, can also be understood as a downward-closed truncation of a nuclear MBE~\eqref{eq:nuclear_mbe}. Note that while the former fragment-MBE truncation can then be exactly expressed as a sum of nuclear-MBE contributions $\sum_{\bm{u} \in I} \tilde{V}_{\bm{u}}$ indexed by the appropriate $I \subseteq B_M$, only the nuclear-MBE subproblem potentials $V_{\bm{u}}$ for $\bm{u} \in \hat{F}$ actually matter in this sum, and the remainder indexed by $B_M - \hat{F}$ can be chosen completely arbitrarily.

Theorem~\ref{thm:combination_consistency} applies to any arbitrary subposet $\hat{F}$ of $B_M$. This can be used to rigorously understand some, and we claim most if not all, of the various fragmentation methods that build from sets of overlapping fragments just as producing particular truncations of a fragment MBE, and therefore transitively an underlying nuclear MBE. This has been informally anticipated, and in some cases either shown directly or argued by example, see again~\cite{collins2015, raghavachari2015, koenig2016}. The following result is hypothesized in~\cite{koenig2016}, but with only a non-rigorous justification. Some further practical demonstration is given in~\cite{hellmers2021}. The proof is condensed and clarified from the version in~\cite[Sec.~5.2]{barker2024}. It should be stressed that this is, in essence, just a variant of the same argument used in~\cite[p.~265]{stanley2012} to construct the set-cardinality PIE using Möbius inversion. Observe also the connection to the decomposition in terms of the meets of an antichain in~\cite{lafuente2005}.

\begin{proposition}\label{prop:gmbe_is_mbe_truncation}
  Let $F = \pdbrace{F_i}_{i=1}^K$ be a not necessarily disjoint fragmentation of $\pdbrack{M}$. Assuming that $V_{\emptyset} = 0$, then for any $1 \leq n \leq K$ there exists an order ideal $I$ of $B_M$ such that $S_I = E^{\text{GMBE}}_{\pdparen{n}}$.
\end{proposition}
\begin{proof}
  Let $\pdbrace{F'_i}_{i=1}^{K'}$ be the set of $n$-mers required for an $n$-body GMBE, built from $F$. Notationally, define $F'_{\cap\bm{u}} \coloneq \bigcap_{i \in \bm{u}} F'_i$ for every $\emptyset \subset \bm{u} \subseteq \pdbrack{K'}$. Take $\hat{F} \coloneq \Set{F'_{\cap\bm{u}} \given \emptyset \subset \bm{u} \subseteq \pdbrack{K'}} \cup \pdbrace{\emptyset}$, cf\@.~\cite[p.~265]{stanley2012}, and rewrite~\eqref{eq:gmbe} as
  \begin{equation}\label{eq:rewritten_gmbe}
    E^{\text{GMBE}}_{\pdparen{n}}
    =
    \sum_{\emptyset \subset \bm{u} \subseteq \pdbrack{K'}}
    \pdparen{-1}^{\abs{\bm{u}}+1} E_{F'_{\cap\bm{u}}}
    =
    \sum_{\substack{\hat{\bm{u}} \in \hat{F}\\\hat{\bm{u}} \neq \emptyset}}
    d_{\hat{\bm{u}}} E_{\hat{\bm{u}}},
  \end{equation}
  where we define $d_{\hat{\bm{u}}}$ to be the sum of all $\pm{}1$ coefficients in terms $\pdparen{-1}^{\abs{\bm{u}}+1}E_{F'_{\cap\bm{u}}}$ where $F_{\cap\bm{u}} = \hat{\bm{u}}$, cf\@., e.g.,~\cite{liu2016}.

  We will show that each $d_{\hat{\bm{u}}} = D^{\pdparen{\hat{F}}}_{\hat{\bm{u}}}$, where the latter is the combination coefficient of $\hat{\bm{u}} \in \hat{F}$ viewed as an order ideal of itself. First, fix some non-empty $\hat{\bm{u}} \in \hat{F}$. By construction, there exists at least one $\bm{u} \subseteq \pdbrack{K'}$ such that $F'_{\cap\bm{u}} = \hat{\bm{u}}$; moreover, there exists exactly one such $\bm{u}$ with maximal $\abs{\bm{u}}$, for if there were two or more, $\bm{u}_1$, $\bm{u}_2$, $\ldots$, it would hold that $\hat{\bm{u}} = F'_{\cap\pdparen{\bm{u}_1\cup\bm{u}_2}}$, contradicting the maximality of $\abs{\bm{u}_1}$ and $\abs{\bm{u}_2}$. Note then that for each $\emptyset \subset \bm{v} \subseteq \bm{u}$, it holds that $F'_{\cap\bm{v}} = \hat{\bm{v}}$ for some $\hat{\bm{v}} \geq_{\hat{F}} \hat{\bm{u}}$, and conversely, each $\hat{\bm{v}} \geq_{\hat{F}} \hat{\bm{u}}$ can be written as $F'_{\cap\bm{v}}$ for at least one $\emptyset \subset \bm{v} \subseteq \bm{u}$. Therefore,
  \begin{equation}\label{eq:sum_of_pie_coeffs_equal_unity}
    \sum_{\hat{\bm{v}} \geq_{\hat{F}} \hat{\bm{u}}} d_{\hat{\bm{v}}}
    =
    \sum_{\emptyset \subset \bm{v} \subseteq \bm{u}} \pdparen{-1}^{\abs{\bm{v}}+1}
    = \sum_{j=1}^{\abs{\bm{u}}} \pdparen{-1}^{j+1} \binom{\abs{\bm{u}}}{j}
    = 1 - \sum_{j=0}^{\abs{\bm{u}}} \pdparen{-1}^{j} \binom{\abs{\bm{u}}}{j}
    = 1,
  \end{equation}
  where we use a standard combinatorial identity, namely~\cite[(1.31)]{quaintance2015}. Finally, by Möbius inversion of~\eqref{eq:sum_of_pie_coeffs_equal_unity} considering the dual $\hat{F}^*$ of $\hat{F}$, that is, $\hat{F}$ reordered by $\bm{u} \geq_{\hat{F}^*} \bm{v}$ if and only if $\bm{u} \leq_{\hat{F}} \bm{v}$, and according to~\eqref{eq:combination_coefficient},
  \begin{equation*}
    d_{\hat{\bm{u}}}
    = \sum_{\hat{\bm{v}} \geq_{\hat{F}} \hat{\bm{u}}} \mu_{\hat{F}^*}\pdparen{\hat{\bm{v}}, \hat{\bm{u}}}
    = \sum_{\hat{\bm{v}} \geq_{\hat{F}} \hat{\bm{u}}} \mu_{\hat{F}}\pdparen{\hat{\bm{u}}, \hat{\bm{v}}}
    = D^{\pdparen{\hat{F}}}_{\hat{\bm{u}}}.
  \end{equation*}

  Since $V_{\emptyset} = 0$, then, it follows that~\eqref{eq:rewritten_gmbe} is just~\eqref{eq:supanova} for this particular $\hat{F}$. Since $\hat{F}$ is closed under intersection by construction, it is a meet subsemilattice of $B_M$, and the desired result follows from Theorem~\ref{thm:combination_consistency} by choosing $I' = \hat{F}$.
\end{proof}

We observe that~\eqref{eq:sum_of_pie_coeffs_equal_unity} coincides with a more general ``top-down'' expression for the FCR coefficients derived by a counting argument in~\cite{hellmers2021}, which in our notation is $D^{\pdparen{I}}_{\bm{u}} = 1 - \sum_{\bm{u} \subset \bm{v} \in I} D^{\pdparen{I}}_{\bm{v}}$. In fact, and more generally again, given any finite order ideal $I$ of a suitable poset $P$, dual-form Möbius inversion (see~\cite[Prop.~3.7.2]{stanley2012}) of~\eqref{eq:combination_coefficient} immediately delivers $D^{\pdparen{I}}_s = 1 - \sum_{s < t \in I} D^{\pdparen{I}}_t$ for arbitrary $s \in I$. This can be understood as a property of the Möbius function, cf\@. Exercise 3.88 of~\cite{stanley2012}.

We required $V_{\emptyset} = 0$ in the statement of Proposition~\ref{prop:gmbe_is_mbe_truncation} in order that the truncation $S_I$ should match the GMBE form~\eqref{eq:gmbe}, and this is more generally the intuitively correct selection if each $V_{\bm{u}}$ should provide the total energy of the subsystem indexed by $\bm{u}$ held in isolation. However, in~\cite{liu2016}, the energy terms $E_{\bm{u}}$ in what can be viewed as a one-body GMBE are calculated by embedding subsystems in a field of electrostatic point charges, and an overcounting correction for the self-interactions of those charges is added to~\eqref{eq:gmbe}, with form justified again via the PIE. This correction can be viewed as a non-zero expression for $V_{\emptyset}$, weighted by $d_{\emptyset} = 1 - \sum_{\emptyset \neq \hat{\bm{u}} \in \hat{F}} d_{\hat{\bm{u}}}$ in the above notation, so a special case of~\eqref{eq:sum_of_pie_coeffs_equal_unity} and the expressions in the previous paragraph.

This, then, is the logic behind our general definition of the subproblem potentials $V_{\bm{u}} : \pdparen{\bbR^3 \times \bbZ}^M \to \bbR$. If, e.g., a family of embedding-style potentials are constructed, such that $V_\emptyset$ represents an extremely coarse model of the full system and each $V_{\bm{u}}$ somehow improves the treatment of the subsystem indexed by $\bm{u}$ compared to those by each $\bm{v} \subset \bm{u}$, then a kind of ``overcounting correction'' analogous to that in~\cite{liu2016} is automatically built into any truncation by way of $D^{\pdparen{I}}_{\emptyset}$. This also applies to truncations of the multilevel decomposition~\eqref{eq:poset_grid_multilevel} below. The idea should find natural application to MBE-like expansions constructed using explicitly quantum embedding techniques; see the review~\cite{jones2020} and cf., e.g., quite recent work in~\cite{veccham2019, schmittmonreal2020}. Although we will not explicitly treat quantum-embedding potentials here, we considered some initial experiments in this direction in~\cite{barker2024}; see there for details and discussion. Although these results were mixed, we still consider this a very promising area for future investigation.

\subsection{Constructing consistent graph-based decompositions}\label{sec:graphs}

Most fragmentation methods exploit either distance-based thresholding criteria or some more abstract idea of structural connectivity in order to exclude some energetic terms from calculation; see again the reviews~\cite{gordon2011, collins2015, raghavachari2015, herbert2019} for many examples. Implicitly or explicitly graph-theoretic techniques can be used to choose an initial fragmentation directly, and/or to derive a decomposition~\cite{collins2006, weiss2010, heber2014, griebel2014, chinnamsetty2018, ricard2020, zhang2021, seeber2023}. For graph-theoretic notation and terminology, we follow mostly~\cite{diestel2017}, with general reference also to~\cite{harary1972, cormen2022}.

Fundamental in this context is the representation of a molecular system via an \emph{interaction graph}, that is, an undirected graph $G = \pdparen{V = \pdbrack{M}, E}$, with an edge set chosen to capture some picture of the connectivity of the individual atoms in the system. Although we have also used $V$ to denote potentials and $E$ energies in the above, the use of standard graph-theoretical notation should be unambiguous in context. Surely the canonical choice for an interaction graph is a \emph{bond graph}, where each edge $\pdbrace{i, j} \in E$ corresponds to a covalent bond between the two atoms with indices $i$, $j$. However, the bond graph is by no means the only plausible interaction graph; see, e.g.,~\cite{weiss2010, ricard2018, zhang2021}.

Once an interaction graph is obtained, subproblems can be identified from particular subgraphs of $G$. A full comparison of the many different ways this can be done is beyond the scope of this article. We build particularly on the BOSSANOVA decomposition~\cite{heber2014, griebel2014}, which is essentially a modified MBE, containing terms only for those subsets $\bm{u} \subseteq \pdbrack{M}$ of nuclear indices which induce connected subgraphs $G\pdbrack{\bm{u}}$ of $G$.

We write the set of all such subsets as $\conn\pdbrack{G}$. For examples, consider the graphs shown in Figure~\ref{fig:benzene_hexane_graphs}. In the hexane graph, leftmost, the induced subgraphs $G\pdbrack{\pdbrace{1, 2, 3}}$ and $G\pdbrack{\pdbrace{1, 4, 5, 6}}$ are connected, so both $\pdbrace{1,2,3}$ and $\pdbrace{1, 4, 5, 6} \in \conn\pdbrack{G}$. But $G\pdbrack{\pdbrace{1, 3, 4}}$ is not, and so $\pdbrace{1, 3, 4}$ is not, respectively.

The BOSSANOVA decomposition of some potential $V$ is
\begin{equation}\label{eq:bossanova}
  V
  = \sum_{\bm{u} \in \conn\pdbrack{G}} \tilde{V}^{\text{BOSSANOVA}}_{\bm{u}},
  \qquad
  \tilde{V}^{\text{BOSSANOVA}}_{\bm{u}} \coloneq
  V_{\bm{u}} - \sum_{\substack{\bm{v} \in \conn\pdbrack{G}\\\bm{v} \subset \bm{u}}} \tilde{V}^{\text{BOSSANOVA}}_{\bm{v}}.
\end{equation}
An informal motivation of BOSSANOVA, consistent with Kohn's famous nearsightedness principle~\cite{prodan2005}, is that at any given order $1 \leq \abs{\bm{u}} \leq M$, the connected subgraphs should provide a tidy way to select the most important $\abs{\bm{u}}$-body interactions in the system. 

Of course,~\eqref{eq:bossanova} is just an instantiation of~\eqref{eq:supanova} for $\hat{F} = \conn\pdbrack{G}$. As mentioned above,~\eqref{eq:supanova} and thus~\eqref{eq:bossanova} can be recognized as special cases of the very general \emph{chemical graph-theoretic cluster expansion} (CGTCE)~\cite{klein1986}. Since the set of all induced subgraphs of $G$ ordered by their inducing sets is clearly isomorphic to $B_M$~\cite{nieminen1980},  it is only a matter of perspective to regard any decomposition of the form~\eqref{eq:supanova} as being determined by a particular collection of induced subgraphs rather than by a particular subposet of $B_M$. Hence, we will refer to any such decomposition as a \emph{SUPANOVA decomposition}, for \emph{SUbgraph Poset ANOVA}. We take the liberty of using a different name since, in context, we treat this setup primarily as a useful extension on the BOSSANOVA decomposition, and we will extend it below to generalize on the multilevel ML-BOSSANOVA scheme~\cite{chinnamsetty2018}. To our reading, such a multilevel adjustment to the CGTCE was not envisioned in~\cite{klein1986}.

\begin{figure}[t]
  {\centering
    \includegraphics{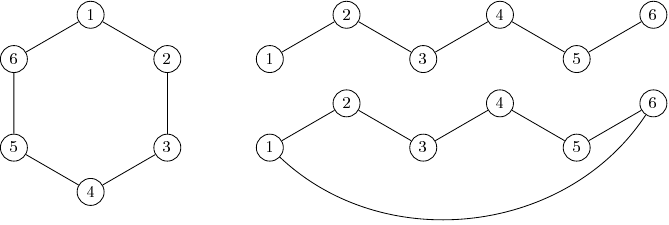}

  }
  \caption{Example fragment interaction graphs for benzene (\ch{C6H6}, left) and hexane (\ch{C6H14}, right). The uppermost graph for hexane shows the covalent bond structure; the lowermost is augmented with a single additional edge.}
  \label{fig:benzene_hexane_graphs}
\end{figure}

Although the results of initial experiments with $n$-body truncations of the BOSSANOVA decomposition over chain-type molecules were highly encouraging~\cite{heber2014, griebel2014}, difficulties arose for molecules containing ring substructures. These were initially ascribed to physical effects in ring fragmentation. While these are important, a more subtle, yet critical issue becomes obvious in light of Theorem~\ref{thm:combination_consistency}.

It is rather easy to find interaction graphs $G$ where $\conn\pdbrack{G}$ is not closed under intersection and therefore cannot be a meet subsemilattice of the appropriate Boolean algebra. This is true for two of the three graphs shown in Figure~\ref{fig:benzene_hexane_graphs}. In the hexane graph, for instance, the intersection $\pdbrace{1, 2, 3, 4} \cap \pdbrace{1, 4, 5, 6} = \pdbrace{1,4}$ does not induce a connected subgraph. So in this case, $\conn\pdbrack{G}$ is not a meet subsemilattice of $B_6$. Rephrasing Theorem~\ref{thm:combination_consistency}, this means that at least one downwards-closed $I'$-truncation of~\eqref{eq:bossanova} has no equivalent downwards-closed $I$-truncation of the MBE~\eqref{eq:nuclear_mbe}. In fact, there are multiple such, and these truncations systematically miscount MBE contribution terms $\tilde{V}_{\bm{u}}$. A worked example demonstrating this effect is given in the supporting information. Interestingly, this seems to be the same issue which necessitates the ``ring repair rule'' in the SMF method~\cite{collins2006}, cf\@.~\cite{collins2012}, although we do not attempt to argue this rigorously.

This is not an isolated case. The conditions on a general graph $G$ under which $\conn\pdbrack{G}$ is a sublattice of $B_M$ were established in~\cite{nieminen1980}. It follows in particular from~\cite[Lems.~1 and 2]{nieminen1980} that $\conn\pdbrack{G}$ cannot be combination-consistent with $B_M$ if $G$ contains certain forbidden induced subgraphs, specifically any chordless cycle of length greater than three, or a cycle of length four with one chord. For details and more precise formulation, see~\cite[Sec.~6.3]{barker2024}. But this rules out the reliable application of the BOSSANOVA decomposition~\eqref{eq:bossanova} in terms of, for instance, the bond graphs of the vast majority of cyclic compounds.

We seek, then, a suitable subposet of subgraphs of $G$ that embody a similar kind of locality to the connected subgraphs, while still being a combination-consistent subposet of $B_M$. One possibility is to use the geodesically convex subgraphs of $G$. In fact, these were mentioned on~\cite[p.~156]{klein1986} as providing a ``well-behaved'' option in a particular application of the CGTCE, namely the construction of a certain coupled cluster-like expansion. We give here only a self-contained definition of the geodesically convex subgraphs of a graph that is essentially the same as that used there. A deeper understanding of these and other kinds of convex subgraphs should be grounded in the theory of abstract convexity; see, e.g.,~\cite{edelman1985} for general background,~\cite{pelayo2013} for a treatment of graph convexities in particular, and~\cite[Sec.~6.5]{barker2024} for more detailed summary and further references. We base our notation on that in, e.g.,~\cite{goddard2011}.

\begin{definition}[Geodesically convex subgraphs of an interaction graph]
  Let $G$ be a connected interaction graph, with vertex set $V = \pdbrack{M}$. An induced subgraph $G\pdbrack{\bm{u}}$ of $G$ is \emph{(geodesically) convex} if, for any pair of vertices $i$, $j \in \bm{u}$, all vertices along any shortest path between $i$ and $j$ in the full graph $G$ lie also in $G\pdbrack{\bm{u}}$. The \emph{geodesic convexity} of $G$, written $\mathcal{M}_g\pdbrack{G}$, is the subposet of $B_M$ formed by all $\bm{u} \in B_M$ such that $G\pdbrack{\bm{u}}$ is a convex subgraph of $G$.
\end{definition}

Although any convex subgraph is a connected subgraph, the inverse is not necessarily true. In the benzene graph in Figure~\ref{fig:benzene_hexane_graphs}, $G\pdbrack{\pdbrace{1, 2, 3}}$ is a convex subgraph, so $\pdbrace{1,2,3} \in \mathcal{M}_g\pdbrack{G}$. But $\pdbrace{1, 4, 5, 6}$ is not, since a shortest path between vertices 1 and 4 passes through vertices 2 and 3.

Combination-consistency of $\mathcal{M}_g\pdbrack{G}$ with $B_M$ results from the obvious fact that if $G\pdbrack{\bm{u}}$ and $G\pdbrack{\bm{v}}$ are convex subgraphs of $G$, then so too is $G\pdbrack{\bm{u} \cap \bm{v}}$. For example, the subposet of $B_6$ shown in Figure~\ref{fig:convex_subgraphs_subposet} is precisely $\mathcal{M}_g\pdbrack{G}$ for the benzene graph, or identically the augmented hexane graph, shown in Figure~\ref{fig:benzene_hexane_graphs}. It was on this basis that we could assert in Section~\ref{sec:inversion} the combination-consistency of the drawn subposet with $B_6$.

However, $\mathcal{M}_g\pdbrack{G}$ is not always an ideal poset for use in a SUPANOVA expansion~\eqref{eq:supanova}. An important result~\cite[Thm.~4.1]{farber1986} shows that $\mathcal{M}_g\pdbrack{G}$ fails in many cases of practical interest to be a \emph{convex geometry}~\cite{edelman1985}. In particular, and informally, subsets $\bm{u} \prec_{\mathcal{M}_g\pdbrack{G}} \bm{v}$ may then have substantially different sizes $\abs{\bm{u}}$ and $\abs{\bm{v}}$; cf.~\cite[Thm.~2.1]{edelman1985}. This has practical implications for the adaptive algorithm we give below. It would be very interesting to investigate whether general schemes could be developed for the construction and/or adjustment of $G$ in a way that would ensure that $\mathcal{M}_g\pdbrack{G}$ is a convex geometry, but this remains for future work.

There is an interesting connection between the decomposition~\eqref{eq:supanova} for $\mathcal{M}_g\pdbrack{G}$, which we will call a \emph{convex SUPANOVA} decomposition, and the energy expression of a multilevel graph-based fragmentation method described in, e.g.,~\cite{ricard2020, zhang2021}. We consider explicitly only a non-multilevel expression given in~\cite{zhang2021}. If an induced subgraph $G\pdbrack{\bm{u}}$ of $G$ is complete, that is, if there is an edge in $G$ between every pair of vertices in $\bm{u}$, then it is called in the language of the scheme a \emph{rank-$r$ simplex}, where $r = \abs{\bm{u}} - 1$. We write $\comp_R\pdbrack{G}$ to collect those subsets of $\pdbrack{M}$ which induce simplices of rank up to some $0 \leq R \leq M-1$. Then the energy expression is, adapting~\cite[(3)]{zhang2021},
\begin{equation}\label{eq:simplex_nonml}
  E^{\text{simplex}}_R =
  \sum_{r=0}^R \pdparen{-1}^r
  \sum_{\substack{\bm{u} \in \comp_R\pdbrack{G}\\\abs{\bm{u}} = r + 1}}
  E_{\bm{u}}
  \pdparen*{
    \sum_{m=r}^R
    \pdparen{-1}^m p^{m}_{\bm{u}}
  },
\end{equation}
where $p^{m}_{\bm{u}}$ counts the number of rank-$m$ simplices which include the rank-$r$ simplex $G\pdbrack{\bm{u}}$ as a subgraph.

As noted for simplices in, e.g.,~\cite{ricard2020}, if $\bm{u} \subseteq \bm{v}$, and $G\pdbrack{\bm{v}}$ is complete, then so too is $G\pdbrack{\bm{u}}$. As a result, $\comp_R\pdbrack{G}$ is an order ideal of $B_M$. The expression~\eqref{eq:simplex_nonml} can be obtained directly by manipulation of~\eqref{eq:mbe_combination_sum} after a corresponding insertion of~\eqref{eq:boolean_algebra_moebius_fn}; indeed, this is just a special case of an expression already given in the FCR setting, see~\cite[(14), (15)]{koenig2016}. So~\eqref{eq:simplex_nonml} is simply a particular truncation of~\eqref{eq:nuclear_mbe}; this was recognized in, e.g.,~\cite{ricard2020, zhang2021}. But note also that every complete induced subgraph of $G$ is also a convex subgraph, although not necessarily vice-versa. As a result, $\comp_R\pdbrack{G}$ is also an order ideal of $\mathcal{M}_g\pdbrack{G}$, and so~\eqref{eq:simplex_nonml} is one truncation of a convex SUPANOVA decomposition.

It is not hard to imagine situations where it would be important to have access to the full range of convex SUPANOVA truncations, rather than only the subset provided by~\eqref{eq:simplex_nonml}. Consider, for instance, a chain system of disjoint fragments, for instance, a linear alkane like that shown upper-right in Figure~\ref{fig:benzene_hexane_graphs}. Here, the only simplices would be the empty graph ($R=-1$), the individual fragments ($R=0$), and neighboring pairs ($R=1$). The full graph $G$ is not complete, so $E^{\text{simplex}}_R$ cannot ``converge'' to the full-system energy as $R$ increases, and, e.g., neighboring triples of fragments are not simplices, so the descriptive power of the simplex approach for longer-range interactions is inherently limited to the edge set of $G$. Were a full convex SUPANOVA expansion considered instead, the adaptive method for truncation selection that we outline below should be able to explore up through the additional terms, and also obviate the need to explicitly preselect a maximum rank $R$ as a parameter.

\subsection{Combining composite and fragmentation methods}\label{sec:ml}

Although the true Born-Oppenheimer potential $V^{\text{BO}}$ can in principle be expanded as, e.g., an MBE like~\eqref{eq:nuclear_mbe}, this is hardly useful, since the potentials $V_{\bm{u}}$ still remain numerically inaccessible. Practically, instead, one must consider an expansion instead of a full-system potential $V_{m,p}$ at a single level of theory. Multilevel fragmentation methods~\cite{dahlke2007, beran2009, rezac2010, bates2011, chinnamsetty2018, ricard2020, hellmers2021, zhang2021} attempt to achieve an effective approximation of such a full-system potential at a high-quality level of theory via evaluations of some of the subproblem potentials $V_{\bm{u}}$ at the same high level, and of others at a lower and thus computationally cheaper level. In this section, we interlace ideas from the previous sections with those found in existing multilevel methods in order to derive a new and very general multilevel fragmentation approach.

By \emph{multilevel}, we refer most generally to any scheme which produces a single approximate solution to~\eqref{eq:schroedinger} from a collection of other and presumably less-accurate solutions calculated from multiple levels of theory. The composite methods mentioned in Section~\ref{sec:fundamentals} clearly form a class of multilevel approach~\cite{zaspel2019}. Another, very well-known such class is provided by the multilayered ONIOM method~\cite{svensson1996} and its many variants and derivatives~\cite{chung2015}. Stated most simply, the principle of ONIOM is to approximate the total energy $E^{\text{HL}}$ of a molecular system as per some high level of theory as
\begin{equation*}
    E^{\text{HL}} \approx E^{\text{LL}} + \pdparen*{E^{\text{HL}}_{\bm{u}} - E^{\text{LL}}_{\bm{u}}},
\end{equation*}
see and cf.~\cite[(3)]{svensson1996}. Here, $E^{\text{LL}}$ is the full-system total energy calculated instead using a lower level of theory, and $E^{\text{HL}}_{\bm{u}}$ and $E^{\text{LL}}_{\bm{u}}$ the high/low-level total energies of a subsystem indexed by $\bm{u} \subset \pdbrack{M}$ and deserving of accentuation. The idea extends to a successively-nested chain of $n$ subsystems $\pdbrack{M} = \bm{u}_1 \supset \bm{u}_2 \supset \cdots \supset \bm{u}_n$, each matched with an increasingly stronger level of theory; see and cf.~\cite[(4)]{svensson1996}. 

The core ONIOM idea has provided a springboard for the development of many multilevel fragmentation methods; see, e.g.,~\cite{bates2011, mayhall2011, ricard2020, zhang2021, seeber2023}, and also again~\cite{gordon2011, chung2015, raghavachari2015, herbert2019}. To connect these and other multilevel fragmentation methods with composite methods, we introduce a very general expansion of some potential $V : \pdparen{\bbR^3 \times \bbN}^M \to \bbR$. We begin by recalling, see, e.g.,~\cite{stanley2012}, that the \emph{direct product} of two posets $P \times Q$ is ordered such that $\pdparen{s, t} \leq_{P \times Q} \pdparen{s', t'}$ iff $s \leq_P s'$ and $t \leq_Q t'$. The following basic result is very helpful; we omit the proof.

\begin{theorem}[Product theorem~\cite{rota1964, stanley2012}]\label{thm:product}
  For the direct product $P \times Q$ of two locally finite posets $P$ and $Q$, it holds that
  \begin{equation}\label{eq:product_thm}
    \mu_{P \times Q}\pdparen{\pdparen{s, t}, \pdparen{s', t'}} =
      \mu_P\pdparen{s, s'} \mu_Q\pdparen{t, t'}.
  \end{equation}
\end{theorem}

Call the members of some family of locally finite posets
$\pdbrace{P_i}_{i=1}^d$ to be poset \emph{axes}, and by extension their direct
product $\Pi = P_1 \times \cdots \times P_d$ to be a $d$-dimensional poset
\emph{grid}. Note that $\mu_{\Pi}$ can be obtained from~\eqref{eq:product_thm} by recursively rewriting, e.g., $\Pi = P_1 \times \pdparen{P_2
  \times \cdots \times P_d}$. The axes $P_i$ need not be finite, but we will
assume them to each have a $\hat{0}$. Then, analogously to above, we assume
the existence of some family of subproblem potentials
$\pdbrace{V_{\bm{p}}}_{\bm{p} \in \Pi}$, and expand some $V : \pdparen{\bbR^3
  \times \bbN}^M \to \bbR$ as
\begin{equation}\label{eq:poset_grid_multilevel}
  V = \sum_{\bm{p} \in \Pi} \tilde{V}_{\bm{p}},
  \qquad
  \tilde{V}_{\bm{p}}
  \coloneq
  V_{\bm{p}} - \sum_{\bm{q} < \bm{p}} \tilde{V}_{\bm{q}}
  =
  \sum_{\bm{q} \leq \bm{p}} \mu_{\Pi}\pdparen{\bm{q}, \bm{p}} V_{\bm{q}}.
\end{equation}
This equality, and convergence of the sum in the case of one or more infinite axes $P_i$, is conditional on the precise choice of subproblem potentials. The idea here is to choose these and also the overall grid $\Pi$ in such a way that pointwise convergence to $V^{\text{BO}}$ holds, practically if not provably. In any case, given any finite order ideal $I$ of $\Pi$, the $I$-truncation $S_I = \sum_{\bm{p} \in I} \tilde{V}_{\bm{p}}$ certainly exists, and can be converted into a sum in terms of combination coefficients $D^{\pdparen{I}}_{\bm{p}}$ and subproblem potentials $V_{\bm{p}}$.

For particular choices of $\Pi$, the expansion form~\eqref{eq:poset_grid_multilevel} captures not only the nuclear MBE~\eqref{eq:nuclear_mbe}, the fragment MBE~\eqref{eq:fragment_mbe}, the  BOSSANOVA decomposition~\eqref{eq:bossanova} and its SUPANOVA (or CGTCE) generalization~\eqref{eq:supanova}, but also, for $\Pi = \pdbrack{2N-1} \times \bbN$, the composite-style expansion~\eqref{eq:bo_composite_expansion}. It can also reproduce a number of existing multilevel fragmentation methods. In the interest of space, we will only explicitly show this for the relatively recent example of the multilevel ML-FCR approach~\cite{hellmers2021}, which is itself an extension on the original FCR method. Note that, in the case of $B_M$, the FCR expression~\eqref{eq:fcr} was completely equivalent to an $I$-truncation $S_I$, and the benefit of our approach was in the supporting theoretical tools provided. In the multilevel setting, however, although the ML-FCR is already very general, we can extend it further.

Only a two-level ML-FCR formulation is explicitly given in~\cite{hellmers2021}. An extension to further layers is suggested but not given. Such an extension seems to be in use in~\cite{seeber2023}, but we restrict ourselves here to the original formulation. Two separate downward-closed sets of potentially-overlapping fragments are considered, which we write as $\pdbrace{\text{FCR}_{\text{HL}}}$ and $\pdbrace{\text{FCR}_{\text{LL}}}$. Generally, one would expect the former to be a subset of the latter, but the involved formulae work correctly even when this relationship is inverted. Adjusting notation to match~\eqref{eq:fcr} above, the ML-FCR energy equation is
\begin{equation}\label{eq:mlfcr}
  E^{\text{ML-FCR}}\pdparen{z_1, \ldots, z_K}
  =
  \sum_{\bm{f}_l \in \pdbrace{\text{FCR}_{\text{HL}}}} p^{\text{HL}}_{\bm{f}_l} E^{\text{HL}}_{\bm{f}_l}\pdparen{\pdbrace{z}_{\bm{f}_l}} 
  +
  \sum_{\bm{f}_l \in \pdbrace{\text{FCR}_{\text{LL}}}} p^{\text{LL}}_{\bm{f}_l} E^{\text{LL}}_{\bm{f}_l}\pdparen{\pdbrace{z}_{\bm{f}_l}},
\end{equation}
see and cf.~\cite[(15)]{hellmers2021}. Each energy term $E^{\text{HL}}_{\bm{f}_l}$ is calculated with a somehow high-level method, and each $E^{\text{LL}}_{\bm{f}_l}$ with a low-level method. The high-level coefficients $p^{\text{HL}}_{\bm{f}_l}$ are non-multilevel FCR coefficients~\eqref{eq:fcr_coefficient} for $\bm{f}_l$ in $\pdbrace{\text{FCR}_{\text{HL}}}$. The high- and low-level coefficients are connected by
\begin{equation}\label{eq:mlfcr_coefficients}
  p^{\text{HL}}_{\bm{f}_l} +
  p^{\text{LL}}_{\bm{f}_l} = p^{\pdbrace{\text{FCR}_{\text{HL}}} \cup
  \pdbrace{\text{FCR}_{\text{LL}}}}_{\bm{f}_l},
\end{equation}
motivated in~\cite{hellmers2021} via a counting argument, where similarly $p^{\pdbrace{\text{FCR}_{\text{HL}}} \cup \pdbrace{\text{FCR}_{\text{LL}}}}_{\bm{f}_l}$ is the non-multilevel FCR coefficient for $\bm{f}_l$ in $\pdbrace{\text{FCR}_{\text{HL}}} \cup \pdbrace{\text{FCR}_{\text{LL}}}$.

To reconstruct the ML-FCR from~\eqref{eq:poset_grid_multilevel}, including a generalization to multiple layers, we choose $\Pi = B_M \times \pdbrack{n}$, for some $n \geq 1$. For an order ideal $I$ of $\Pi$, we can write $I = \bigcup_{i=1}^n I_i \times \pdbrace{i}$, where each $I_i$ is an order ideal of $B_M$ and also $I_1 \supseteq I_2 \supseteq \cdots \supseteq I_n$. Recalling~\eqref{eq:combination_coefficient} and using~\eqref{eq:product_thm}, we have that
\begin{equation}\label{eq:i_eq_n_ml_coeffs}
  D^{\pdparen{I}}_{\pdparen{\bm{u}, n}}
  = \sum_{\substack{\bm{v} \supseteq \bm{u}\\\bm{v} \in I_n}} \mu_{\Pi}\pdparen{ \pdparen{\bm{u}, n}, \pdparen{\bm{v}, n} }
  = \sum_{\substack{\bm{v} \supseteq \bm{u}\\\bm{v} \in I_n}} \mu_{B_M}\pdparen{ \bm{u}, \bm{v} }
  = D^{\pdparen{I_n}}_{\bm{u}},
\end{equation}
where the final $D^{\pdparen{I_n}}_{\bm{u}}$ is a combination coefficient for an $I_n$-truncation of a sum indexed from $B_M$. Clearly $D^{\pdparen{I}}_{\pdparen{\bm{u}, i}} = 0$ when $i > n$. When $i < n$,
\begin{equation}\label{eq:i_lt_n_ml_coeffs}
  \begin{aligned}
  D^{\pdparen{I}}_{\pdparen{\bm{u}, i}}
  &=
  \sum_{j=i}^n \sum_{\substack{\bm{v} \supseteq \bm{u}\\\bm{v} \in I_j}} \mu_{B_M}\pdparen{\bm{u}, \bm{v}} \mu_{\pdbrack{n}}\pdparen{i,j}
  \\
  &=
  \sum_{\substack{\bm{v} \supseteq \bm{u}\\\bm{v} \in I_i}} \mu_{B_M}\pdparen{\bm{u}, \bm{v}}
  -
  \sum_{\substack{\bm{v} \supseteq \bm{u}\\\bm{v} \in I_{i+1}}} \mu_{B_M}\pdparen{\bm{u}, \bm{v}}
  =
  D^{\pdparen{I_i}}_{\bm{u}} - D^{\pdparen{I_{i+1}}}_{\bm{u}},
  \end{aligned}
\end{equation}
then~\eqref{eq:mlfcr} emerges by fixing $n = 2$, and choosing $I_1 = \pdbrace{\text{FCR}_{\text{LL}}}$ and $I_2 = \pdbrace{\text{FCR}_{\text{HL}}}$. The coefficients $p^{\text{HL}}_{\bm{f}_l}$ and $p^{\text{LL}}_{\bm{f}_l}$ are  recovered from~\eqref{eq:i_eq_n_ml_coeffs} and~\eqref{eq:i_lt_n_ml_coeffs}, cf\@.~\cite[(17), (18)]{hellmers2021}, and these also give a rederivation of~\eqref{eq:mlfcr_coefficients}.

Other choices of $I_1$ and $I_2$, or, viewed alternatively, $\pdbrace{\text{FCR}_{\text{LL}}}$ and $\pdbrace{\text{FCR}_{\text{HL}}}$, lead, after some straightforward manipulation and sometimes with recourse to the argument in Proposition~\ref{prop:gmbe_is_mbe_truncation}, also to the working equations of a variety of other multilevel fragmentation methods. These include a multilevel version of the EE-MB~\cite{dahlke2007}, the HMBI~\cite{beran2009}, the MFBA~\cite{rezac2010}, and the MC QM/QM method~\cite{bates2011}. The full multilevel form of the simplex-based method in~\cite{ricard2020, zhang2021} can also be recovered via the discussion of~\eqref{eq:simplex_nonml} above, and for more general $n$, the MIM~\cite{mayhall2011} emerges. We omit the details for reasons of space, but see~\cite[Sec.~7.2]{barker2024}. Equivalences between some of these methods have certainly been mentioned before, as in, e.g.,~\cite{beran2009, hellmers2021}, but this provides further support for the generality of the ML-FCR approach, at least for existing fragmentation methods.

Alternatively, given an interaction graph $G$, choosing in~\eqref{eq:poset_grid_multilevel} instead $\Pi = \conn\pdbrack{G} \times \bbN$ delivers the ML-BOSSANOVA extension~\cite{chinnamsetty2018} to the original BOSSANOVA scheme. ML-BOSSANOVA considers a family of potentials $\pdbrace{V_{p}}_{p \in \bbN}$, where each $p$ indicates the use of a progressively better-quality basis set than for $p-1$. The potentials $V_{p}$ are assumed to converge to the true $V^{\text{BO}}$ as $p \to \infty$. Each $V_p$ is decomposed as $V_p = \sum_{\bm{u}\in\conn\pdbrack{G}} \tilde{V}_{p,\bm{u}}$, with $\tilde{V}_{p,\bm{u}}$ defined equivalently as in~\eqref{eq:bossanova}. Each $\tilde{V}_{p,\bm{u}}$ is then itself decomposed as
\begin{equation}
  \tilde{V}_{p,\bm{u}} = \sum_{q=0}^p \tilde{\omega}_{q,\bm{u}},
  \quad
  \text{with}
  \quad
  \tilde{\omega}_{q,\bm{u}} = \begin{cases}
    \tilde{V}_{q,\bm{u}} - \tilde{V}_{q-1,\bm{u}}&\text{when $q \geq 1$, and}\\
    \tilde{\omega}_{0,\bm{u}} = \tilde{V}_{0,\bm{u}}&\text{as a special case.}
  \end{cases}
\end{equation}
A downwards-closed truncation of the resulting exact decomposition
\begin{equation}\label{eq:mlbossanova}
  V^{\text{BO}} =
  \sum_{\bm{u} \in \conn\pdbrack{G}}\sum_{p \in \bbN} \tilde{\omega}_{p,\bm{u}}
\end{equation}
can be grown adaptively, and should, informally, capture the non-negligible terms from an expansion of $V^{\text{BO}}$ constructed similarly, but over the full $B_M$ rather than just $\conn\pdbrack{G}$. This latter would be, effectively, an infinitely multilevel MBE.

\begin{figure}[t]
  {\centering
    \includegraphics{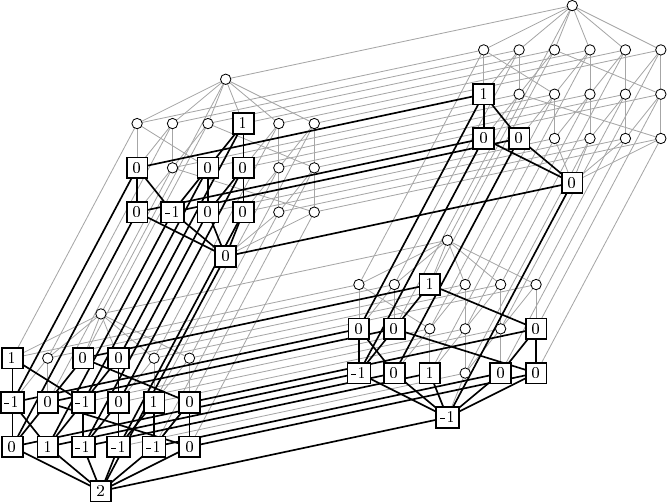}

  }
  \caption{A Hasse-like diagram of the poset grid $\Pi = \mathcal{M}_g\pdbrack{G} \times \pdbrack{2} \times \pdbrack{2}$, for $G$ either the benzene or augmented hexane interaction graphs in Figure~\ref{fig:benzene_hexane_graphs}. Each vertex corresponds to a term $\tilde{V}_{\pdparen{\bm{u}, m, p}}$ in an ML-SUPANOVA expansion like~\eqref{eq:ml_supanova_expansion_explicit}, but restricted to the convex subgraphs of $G$, and using \emph{ab initio} treatments with, e.g., HF and MP2, and a pair of basis sets like, e.g., cc-pVDZ and cc-pVTZ. Rectangular vertices indicate elements and combination coefficients of a particular order ideal of this grid, or equivalently, the terms of a corresponding downwards-closed truncation of the expansion.}
  \label{fig:ml_supanova_example_grid}
\end{figure}

We extend on ML-BOSSANOVA in two distinct ways. Firstly: like all of the other multilevel methods discussed above, ML-BOSSANOVA considers only a single axis for the level of theory. However, using~\eqref{eq:poset_grid_multilevel}, it is straightforward to introduce arbitrarily many axes. This can be motivated as approximating each $\tilde{V}_{\bm{u}}$ in an MBE of $V^{\text{BO}}$ in the style of a composite method, following the combination-technique understanding in~\cite{chinnamsetty2018, zaspel2019}. The starting point is thus the Boolean algebra $B_M$ corresponding to an ANOVA-like decomposition of $V$. The next axis, $P_2 = \pdbrack{2N-1}$, indexes the various levels of \emph{ab initio} theory outlined in~\eqref{eq:qc_method_hierarchy} above. Then we introduce a third axis $P_3 = \bbN$, which indexes members of a systematically-improving family of basis sets, as in ML-BOSSANOVA. Again under appropriate assumptions regarding convergence, the resulting poset grid $\Pi = B_M \times P_2 \times P_3$ produces an exact expansion of the Born-Oppenheimer potential as
\begin{equation}\label{eq:ml_supanova_expansion_explicit}
  V^{\text{BO}}
  = \sum_{\bm{p} \in \Pi} \tilde{V}_{\bm{p}}
  =
  \sum_{\bm{u} \in B_M}
  \sum_{m = 1}^{2N-1}
  \sum_{p=1}^\infty
  \tilde{V}_{\pdparen{\bm{u}, m, p}}.
\end{equation}
The first axis can of course be restricted to a subposet of $B_M$, or isomorphically, a subposet of induced subgraphs of some $G$. Taking the subgraph perspective, we refer to~\eqref{eq:ml_supanova_expansion_explicit} as the most general kind of \emph{ML-SUPANOVA decomposition}.

Secondly: the use of $\conn\pdbrack{G}$ in~\eqref{eq:mlbossanova} can be problematic, much as in the original BOSSANOVA scheme. It is easy to see that a direct product $P \times Q$ is a meet semilattice if and only if both $P$ and $Q$ are. So ML-BOSSANOVA's choice of $\Pi = \conn\pdbrack{G} \times \bbN$ may not be combination-consistent with $B_M \times \bbN$, depending on the structure of $G$, and particularly when $G$ contains cycles such as might be introduced by ring substructures of a molecule. The solution is the same: choose a more appropriate subposet of $B_M$, and thus, for~\eqref{eq:ml_supanova_expansion_explicit}, of $\Pi = B_M \times \pdbrack{2N-1} \times \bbN$. We suggest as a starting point using the geodesic complexity $\mathcal{M}_g\pdbrack{G}$ to produce a \emph{convex ML-SUPANOVA decomposition}; every downwards-closed truncation of this decomposition is then guaranteed to have a counterpart truncation of the full~\eqref{eq:ml_supanova_expansion_explicit}. An illustration of the poset underlying a sample ML-SUPANOVA decomposition, restricted to $\Pi = \mathcal{M}_g\pdbrack{G} \times \pdbrack{2} \times \pdbrack{2}$, is shown in Figure~\ref{fig:ml_supanova_example_grid}. Note that in this particular case, the restriction of the basis-set and \emph{ab initio} axes would make the expansion technically one of $V_{\pdbrack{6}, 2, 2}$, and so only an approximation to $V^{\text{BO}}$.

\section{Adaptive Algorithm}\label{sec:adaptive}

Standard $n$-body truncations of MBEs become for increasing $n$ at best inefficient, and at worst numerically unstable~\cite{richard2014, herbert2019}, and so choosing the correct value of $n$ for use in practice is difficult. When considering expansions like~\eqref{eq:supanova} and especially multilevel forms like~\eqref{eq:ml_supanova_expansion_explicit}, the task of choosing suitable truncations \emph{a priori} becomes even harder again. Algorithm~\ref{alg:adaptive} outlines instead an adaptive, \emph{a posteori} approach for the selection of a truncation order ideal $I$ for a sum~\eqref{eq:poset_grid_multilevel} over some poset grid $\Pi$. This algorithm is a specialization of one given in a more general setting in~\cite[Chap.~3]{barker2024}, and closely related also to that in~\cite{chinnamsetty2018}. The basic idea of the algorithm goes back to work in dimensionally-adaptive numerical quadrature~\cite{gerstner2003}, and there are many other strong connections in the context of sparse grid approximation and the combination technique; see, e.g.,~\cite{hegland2003, bungartz2004, hegland2007}. We observe with interest that a similar approach is taken in the ``bottom-up'' energy-screening algorithm described very recently in~\cite{broderick2023}, which walks a directed acyclic graph constructed for the terms of the GMBE. As there noted, this graph allows interpretation as a Hasse diagram of subsets, and Algorithm~\ref{alg:adaptive} can likewise be interpreted as walking the Hasse diagram of $\Pi$, in the more general order-theoretic sense; cf., e.g.,~\cite{hegland2007}. Our algorithm targets more general expansions than just the GMBE, and differs particularly in the screening criterion applied and the update mechanism; cf\@.~\cite[(12), (13)]{broderick2023}.

\begin{algorithm}[t]
  \begin{algorithmic}[1] % <- Line numbering on
    \State $I, A \gets \emptyset$
    \State $Q \gets \text{an empty priority queue}$
    \State $D, L, E \gets \text{empty (zero) tensors}$
    \State $i \gets -1$
    %\Statex
    \Repeat
    \State $i \gets i + 1$
    %\Statex
    \If{$i = 0$}
    \State $I_{\text{new}} \gets \pdbrace{\hat{0}_{\Pi}}$    
    \Else
    \State $I_{\text{new}} \gets \text{new elements from entries of $Q$ according to selection strategy}$
    \label{algline:select_elements}
    \EndIf
    \State $I \gets I \cup I_{\text{new}}$
    %\Statex
    \ForAll{$\bm{p} \in I_{\text{new}}$}
    \label{algline:new_element_loop}    
    \State $L_{\bm{p}} \gets \mathcal{L}\pdbrack{V_{\bm{p}}}$
    \label{algline:evaluate_potentials}

    %\Statex
    \State Evaluate Möbius tensor $M^{\pdparen{\bm{p}}}$.
    \State $D \gets D + M^{\pdparen{\bm{p}}}$
    \Comment{Update the full combination tensor,}
    \State $E \gets E + M^{\pdparen{\bm{p}}}$
    \Comment{Update the error-indicator tensor.}
    \label{algline:update_error_indicator}
    \State Insert $\bm{p}$ into $Q$, keyed by $\norm{\mathcal{L}\pdbrack{\tilde{V}_{\bm{p}}} \pdparen*{= \Call{Reduce}{M^{\pdparen{\bm{p}}} \odot L}}} / \mathcal{C}\pdparen{\bm{p}}$.
    \label{algline:insert_into_queue}
    \EndFor
    
    %\Statex
    \LComment{Update generating antichain, and correct the error indicator tensor}    
    \State $A \gets A \cup I_{\text{new}}$
    \label{algline:update_generating_antichain}
    \State $R \gets A \cap \bigcup_{\bm{p} \in I_{\text{new}}} \Set{\bm{q} \in \Pi \given \bm{q} \prec \bm{p}}$
    \State $A \gets A - R$
    \State $E \gets E - \sum_{\bm{p} \in R} M^{\pdparen{\bm{p}}}$
    \label{algline:discount_from_error}

    %\Statex
    \LComment{Calculate approximation, error indicator, cost, and uncertainty}
    \State $S_i \gets \Call{Reduce}{D \odot L}$
    \label{algline:calculate_sum}
    \State $\mathcal{E}_i \gets \Call{Reduce}{E \odot L}$
    \label{algline:calculate_error_indicator}
    \State $\mathcal{C}_i \gets \sum_{\bm{p} \in I} \mathcal{C}\pdparen{\bm{p}}$
    \label{algline:calculate_cumulative_abstract_cost}
    \State $\text{d}S_i \gets \sqrt{\Call{Reduce}{\varepsilon^2\pdparen{D \odot D}}}$
    \label{algline:calculate_propagated_uncertainty}
    \Until{$Q$ is empty, or termination criteria are met.}
    
    %\Statex
    \State \Return $I$, and also $S_i$, $\mathcal{E}_i$, $\mathcal{C}_i$, and/or $\text{d}S_i$ as required
  \end{algorithmic}
  \caption{Adaptive algorithm for calculation of order ideal $I$ (adapted from~\cite[Alg.~3.3]{barker2024})}\label{alg:adaptive}
\end{algorithm}

We assume here the choice of a linear operator $\mathcal{L}$, e.g., a point evaluation operator as in Section~\ref{sec:fundamentals}. Sequences of order ideals $I^{\pdparen{i}}$ and their corresponding evaluated sums $S_i \coloneq \mathcal{L}\pdbrack{S_{I^{\pdparen{i}}}}$ are generated, starting from the initial values $I^{\pdparen{0}} = \pdbrace{\hat{0}}$ and $S_{0} = \mathcal{L}\pdbrack{\tilde{V}_{\hat{0}}} = \mathcal{L}\pdbrack{V_{\hat{0}}}$. At the $i$th iteration of the algorithm, at least one element $\bm{p} \in I^{\pdparen{i-1}}$ is selected for \emph{expansion}: all those elements $\bm{q} \succ \bm{p}$ such that $I^{\pdparen{i}} \cup \pdbrace{\bm{q}}$ would remain an order ideal are adjoined to $I^{\pdparen{i-1}}$ to form $I^{\pdparen{i}}$. The contribution potentials $\tilde{V}_{\bm{q}}$ for these new element(s) are evaluated, and accumulated into $S_i$.

The choice of element(s) $\bm{p} \in I^{\pdparen{i-1}}$ for expansion is guided by the benefit/cost ratios $\norm{\mathcal{L}\pdbrack{\tilde{V}_{\bm{p}}}} / \mathcal{C}\pdparen{\bm{p}}$, where $\mathcal{C} : \Pi \to \bbR$ is an \emph{abstract cost model} for the evaluation of the subproblem potentials $V_{\bm{p}}$ by $\mathcal{L}$. One possible abstract cost model for point evaluation is outlined in Appendix~\ref{sec:abstract_cost_model}. Under certain assumptions, see~\cite{bungartz2004, chinnamsetty2018} and~\cite[Sec.~3.4]{barker2024}, a truncation constructed by assembling terms in descending order of benefit/cost leads to a quasi-optimal truncation of~\eqref{eq:poset_grid_multilevel}.

Notationally, we assume that the elements of a $d$-dimensional tensor, i.e., a multidimensional array $T$ can be indexed $T_{\bm{m}}$ by multiindices $\bm{m} \in \bbN^d$. For a poset grid $\Pi = P_1 \times \cdots \times P_d$ where each $P_i$ is countable, we assume the existence of a suitable indexing bijection $\Phi: \Pi \to \bbN^d$ and will write by abuse of notation $T_{\bm{p}}$ to mean the element $T_{\Phi\pdparen{\bm{p}}}$. Elementwise additions ($+$) and multiplications ($\odot$) of tensors behave as usual, and the tensor product $T \otimes U$ of a $k$-dimensional $T$ and $l$-dimensional $U$ is $(k+l)$-dimensional and defined elementwise by $\pdparen{T \otimes U}_{\pdparen{m_1, \ldots, m_k, n_1, \ldots, n_l}} \coloneq T_{\pdparen{m_1, \ldots m_k}}U_{\pdparen{n_1, \ldots, n_l}}$. We write $\textsc{Reduce}\pdparen{T} \coloneq \sum_{\bm{m} \in \Pi} T_{\bm{m}}$. This sum is here always finite, since although tensors may be notionally infinite, those considered have finitely many non-zero terms by construction. For the same reason, it is desirable to work in the implementation with sparse tensors~\cite{bader2008, sparse2018}, that is, higher-dimensional analogues of standard sparse matrices. We omit the details here.

It is a deliberate feature of this notation that the combination coefficients $D^{\pdparen{I^{\pdparen{i}}}}_{\bm{p}}$ as per~\eqref{eq:combination_coefficient} coincide identically with the entries of the \emph{combination tensor}
\begin{equation*}
  D^{\pdparen{I^{\pdparen{i}}}} \coloneq \sum_{\bm{p} \in I^{\pdparen{i}}} M^{\pdparen{\bm{p}}},
\end{equation*}
where each $M^{\pdparen{\bm{p}}}$ is the \emph{Möbius tensor} of $\bm{p} \in \Pi$ and is defined elementwise by $M^{\pdparen{\bm{p}}}_{\bm{q}} \coloneq \mu_{\Pi}\pdparen{\bm{q}, \bm{p}}$. As elements $\bm{p} \in \Pi$ are added to the order ideals $I^{\pdparen{i}}$, their evaluations $\mathcal{L}\pdbrack{V_{\bm{p}}}$ are stored in a suitable tensor as $L_{\bm{p}}$. Their Möbius tensors can be used to directly evaluate $\mathcal{L}\pdbrack{\tilde{V}_{\bm{p}}} = \textsc{Reduce}\pdparen{M^{\pdparen{\bm{p}}} \odot L}$, and then accumulated into a running combination tensor $D$. 

Maintaining the sum $S_i$ indirectly via the combination tensor has some appealing numerical characteristics that help to ameliorate issues like those noted in, e.g.,~\cite{richard2014}. A sparsity-exploiting evaluation $S_{i} = \textsc{Reduce}\pdparen{D \odot L}$ can use as few arithmetic operations as possible, whereas a naïve accumulation of evaluated terms $\mathcal{L}\pdbrack{\tilde{V}_{\bm{p}}}$ may suffer over iterations from accrued error due to floating-point arithmetic. Moreover, all \emph{ab initio} methods reduce to the solution of particular systems of equations, which are solved to within some tolerance usually chosen well outside floating-point precision. In~\cite{richard2014}, a propagation-of-errors technique is used to test how these tolerances flow through to the overall uncertainty of a standard $n$-body truncation, rewritten for the purpose in a closed-form expression~\cite[(4.2) and (4.5)]{richard2014}. If we associate to each $\mathcal{L}\pdbrack{V_{\bm{p}}}$ an inherent uncertainty $\varepsilon_{\bm{p}}$, then it is an easy generalization of the same approach to calculate the propagated uncertainty in the sum $S_I$ for any $I$ as
\begin{equation*}%\label{eq:uncertainty}
  \text{d}S_I = \sqrt{
    \sum_{\bm{p} \in I} \pdparen*{D^{\pdparen{I}}_{\bm{p}}}^2 \varepsilon^2_{\bm{p}}
  }.
\end{equation*}

The precise encoding of poset grid elements $\bm{p} \in \Pi$ is an implementation detail. The algorithm requires only that all covered elements $\bm{q} \prec \bm{p}$ can be enumerated, and similarly all covering elements $\bm{r} \succ \bm{p}$. It suffices if the same is possible for any element of an individual axis $p \in P_i$. This is easy in the cases of $\pdbrack{n}$, $\bbN$, and $B_n$. A little more work is required for $\conn\pdbrack{G}$ and particularly $\mathcal{M}_g\pdbrack{G}$; see~\cite[App.~2]{barker2024} for details and supporting references. The values $\mu_\Pi\pdparen{\bm{q}, \bm{p}}$ and thus $M^{\pdparen{\bm{p}}}$ can then always be evaluated explicitly at least as per~\eqref{eq:moebius_function}, which is amenable to memoization. Particularly when expressions like~\eqref{eq:chain_poset_moebius} and~\eqref{eq:boolean_algebra_moebius_fn} are available, it is more convenient to apply~\eqref{eq:product_thm} and write
\begin{equation*}
  M^{\pdparen{\bm{p} = \pdparen{p_1, \ldots, p_d}}} = \bigotimes_{i=1}^d m^{\pdparen{p_i}},
\end{equation*}
using Möbius ``vectors'' defined elementwise by $m^{\pdparen{p_i}}_{s} \coloneq \mu_{P_i}\pdparen{s, p_i}$ for $s, p_i \in P_i$. For example, if $P_i$ is a chain poset, then constructing $m^{\pdparen{p_i}}$ requires no more than two elementwise updates to an initially-zero sparse vector.

\begin{figure}[t]
  {\centering
    \includegraphics{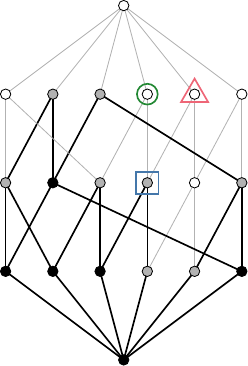}

  }
  \caption{A possible state of Algorithm~\ref{alg:adaptive} at the beginning of an iteration, during an adaptive calculation for the poset displayed in Figure~\ref{fig:convex_subgraphs_subposet}, i.e.,  $\mathcal{M}_g\pdbrack{G}$ for the benzene/augmented hexane graph in Figure~\ref{fig:benzene_hexane_graphs}. This particular state is hypothetical, and constructed for illustration. The order ideal constructed by previous iterations consists of all shaded vertices. Vertices shaded gray indicate elements in the priority queue. Vertices shaded black have already been fully expanded. Suppose that the element indicated by a blue square were considered for expansion. Only its covering element indicated by a green circle is admissible; its other covering element (red triangle) is inadmissible, since it covers itself another element that is not in the order ideal.}
  \label{fig:adaptive_algorithm_example}
\end{figure}

We maintain throughout the course of the algorithm a priority queue, see~\cite{cormen2022} and cf\@.~\cite{gerstner2003}, of all elements $\bm{p} \in \Pi$ which have been previously added to $I$ and not yet fully expanded, in the sense that not all elements $\bm{q} \succ \bm{p}$ have also been added to $I$. The queue is kept in descending order of $\norm{\mathcal{L}\pdbrack{\tilde{V}_{\bm{p}}}} / \mathcal{C}\pdparen{\bm{p}}$. At line~\ref{algline:select_elements}, elements $\bm{p}$ are removed one-by-one from the front of the queue, and their covering elements $\bm{r} \succ \bm{p}$ are tested for \emph{admissibility}, which holds whenever all elements $\bm{p}' \prec \bm{r}$ are also already in $I$. Any such admissible $\bm{r}$ joins the overall set of new elements $I_{\text{new}}$, such that $I^{\pdparen{i}} \leftarrow I^{\pdparen{i-1}} \cup I_{\text{new}}$. A $\bm{p}$ with at least one admissible $\bm{r} \succ \bm{p}$ is \emph{expandable}. An example showing the difference between admissible and inadmissible covering elements of an expandable $\bm{p}$ is given in Figure~\ref{fig:adaptive_algorithm_example}.

Different strategies can be applied to guide the number of elements expanded at each step. One might simply stop after expanding the first, \textsc{Best} expandable element in the queue, or instead consider \textsc{All} expandable elements in the queue. A more flexible strategy is to continue through the queue, expanding any elements whose benefit/cost ratio is within some \textsc{Threshold} factor $0 \leq \alpha \leq 1$ of that of the best. Whichever strategy is used, any dequeued element found to have an inadmissible $\bm{r} \succ \bm{p}$ should be requeued.

Once the new elements $\bm{p} \in I_{\text{new}}$ have been identified, each corresponding $\mathcal{L}\pdbrack{V_{\bm{p}}}$, $M^{\pdparen{\bm{p}}}$, $\mathcal{L}\pdbrack{\tilde{V}_{\bm{p}}}$, and $\mathcal{C}\pdparen{\bm{p}}$ must be evaluated, the combination tensor $D^{\pdparen{I}}_{\bm{p}}$ updated, and $\bm{p}$ inserted into the queue. This occurs in the loop at line~\ref{algline:new_element_loop}, along with one additional step which will be explained below. Parallelism in this loop should be exploited as much as possible, since in practice, the evaluations $\mathcal{L}\pdbrack{V_{\bm{p}}}$ are by far the most costly part of the algorithm.

The algorithm terminates once either the queue has been exhausted, or some termination criterion has been met. One might place a cumulative limit on the abstract costs $\mathcal{C}\pdparen{\bm{p}}$, but it is more useful to have an estimating \emph{error indicator}, see and cf., e.g.,~\cite{gerstner2003, bungartz2004}. For this, we suggest
\begin{equation}\label{eq:error_indicator}
  \mathcal{E}_i = \sum_{\bm{p} \in A^{\pdparen{i}}} \mathcal{L}\pdbrack{\tilde{V}_{\bm{p}}},  
\end{equation}
where $A^{\pdparen{i}}$ is the set of maximal elements of $I^{\pdparen{i}}$. For efficiency, the adaptive algorithm also maintains and updates this set. The key observation is that any new $\bm{p}$ added to $I^{\pdparen{i}}$ must be a maximal element of the same, and that any elements $\Set{\bm{q} \in I^{\pdparen{i-1}} \given \bm{q} \prec \bm{p}}$ are correspondingly precluded from maximality in $I^{\pdparen{i}}$. Again, we calculate~\eqref{eq:error_indicator} not as a sum of contribution terms, but rather maintain a set of combination coefficients analogous to those for the full sum (lines~\ref{algline:update_error_indicator} and~\ref{algline:update_generating_antichain}--\ref{algline:discount_from_error}).

\section{Experiments and Discussion}\label{sec:experiments}

As an initial evaluation of the techniques outlined above, we consider a collection of experimental results obtained by adaptively growing truncated sums of the form~\eqref{eq:poset_grid_multilevel}. These results are a subset of those covered in~\cite{barker2024}, and were generated using an implementation of the more general version of Algorithm~\ref{alg:adaptive} given there.

A detailed discussion of implementation details is not necessary to interpret these results, and is left to~\cite{barker2024}. Similarly out of scope is a discussion of the computational treatments of the various involved posets; for this, see particularly Appendix B of that source. Briefly, however, the implementation made use of standard scientific libraries~\cite{harris2020, virtanen2020}, particularly NetworkX~\cite{hagberg2008} for working with graphs. Some reference was made to the Sparse library for sparse arrays~\cite{sparse2018}. In general, tensor reductions were performed using interval arithmetic as implemented in the Arb library~\cite{johansson2017}, using 100 bits of precision, but we believe this had no significant impact on results compared to the use of standard double-precision floating-point arithmetic. Some of the data in Figure~\ref{fig:heptane_results} were regenerated to correct an error in~\cite{barker2024}, and here, standard double-precision was used. Plots and visualizations are colored using the ``bright'' scheme of~\cite{tol2021}, while molecular visualizations were rendered in Blender using the DSO shading scheme in~\cite{hansen2020}.

Results given here were obtained using HF~\cite{szabo1996}, MP2~\cite{moeller1934}, CCSD~\cite{cizek1966, purvis1982}, and CCSD(T)~\cite{raghavachari1989} calculations; KS-DFT/B3LYP~\cite{stephens1994} was used for geometry optimizations. We used variously the following basis sets: STO-3G\cite{hehre1969, hehre1970}, cc-pVDZ~\cite{dunning1989, woon1993}, cc-pVTZ, cc-pCV$n$Z for $2 \leq n \leq 8$~\cite{dunning1989, woon1995, wilson1996, peterson1997, feller1999, mielke1999, feller2000, feller2007, feller2008, feller2009, feller2010, thorpe2021}, and also all necessary basis sets for G4(MP2)~\cite{ditchfield1971, hehre1972, hariharan1973, curtiss1998, curtiss2007, curtiss2007a}. Most basis sets were obtained from the Basis Set Exchange~\cite{pritchard2019}, and processed for computational efficiency using the BSE Python library. For particularly fine details on basis sets and, e.g., detailed solver settings, see again~\cite{barker2024}.

Calculations required for subproblem potential evaluation were performed using PySCF~\cite{sun2015, sun2018, sun2020}, to an iterative convergence threshold set at \qty{1e-08}{\hartree}, with a suitably tight integral prescreening threshold. Full-system geometry optimizations, as well as one reference calculation, were performed using NWChem~\cite{apra2020}. Monatomic energies used to calculate atomization energies, see below, were mostly obtained using MRCC~\cite{kallay2020, mrcc, rolik2013, kallay2014, gyevinagy2020}, with HF and CC convergence thresholds at \qty{1e-10}{\hartree}; those for oxygen with aug-cc-pCV8Z were obtained using PySCF, converged to \qty{1e-9}{\hartree}. An implicit assumption here is that single-point results for identical geometries and levels of theory but carefully obtained under different solver packages will be reliably equivalent up to convergence thresholds.

In the following, we are primarily interested in observing changes in the accuracy of the sums at each iteration against their total abstract costs of calculation $\mathcal{C}_i$, as per line~\ref{algline:calculate_cumulative_abstract_cost} of Algorithm~\ref{alg:adaptive} and the abstract cost model outlined in Appendix~\ref{sec:abstract_cost_model}. With one exception, discussed below, the final plots given here were generated by reusing previously-cached subproblem potential evaluations. Since these plots consider only abstract cost, this reuse has no impact on the results as shown.

\subsection{Convex SUPANOVA decomposition}\label{sec:graph_experiments}

Before considering the full multilevel setup, we first investigate the suitability of the convex SUPANOVA correction to the non-multilevel BOSSANOVA method. Since we already know from, e.g.,~\cite{heber2014, griebel2014} that the conventional BOSSANOVA approach is well capable of handling smaller chain-like molecules, we investigate two examples of larger molecular systems with characteristics that would be problematic in light of BOSSANOVA's issues with combination-consistency.

\begin{figure}[t]
  {\centering
    \includegraphics[width=0.8\textwidth]{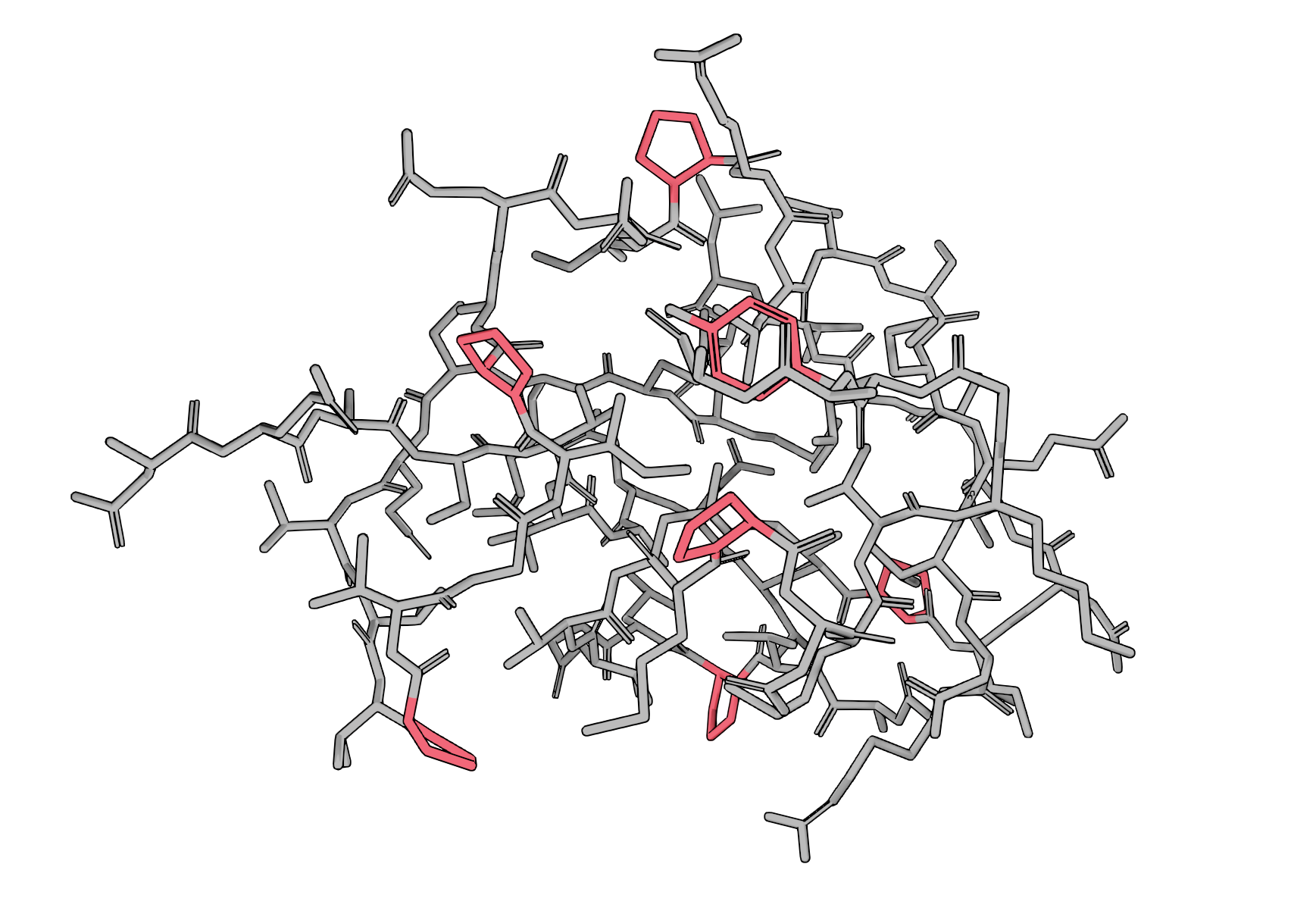}
    
  }
  \caption{Stick-model visualization of the covalent bond structure of the antifreeze protein 1KDF~\cite{soennichsen1996, soennichsen2021}, after preprocessing as described in the main text and with hydrogen atoms excluded. Ring substructures are highlighted in red.}
  \label{fig:1kdf_visualization}
\end{figure}

We begin with a model of an \emph{antifreeze protein}~\cite{soennichsen1996}, obtained from the Protein Data Bank~\cite[PDB key: 1KDF]{soennichsen2021}. As for most proteins, and as indicated in the visualization in Figure~\ref{fig:1kdf_visualization}, the covalent bond structure of 1KDF contains several rings. The 1KDF model was preprocessed using Open Babel~\cite{oboyle2011}, with hydrogens added explicitly. The resulting system contained a total of \num{991} atoms, \num{479} non-hydrogen. Using NWChem, via a calculation initialized by basis set projection from an initial HF/STO-3G calculation and then converged to \qty{1e-9}{\hartree}, we obtained a reference HF/cc-pVTZ total energy of the 1KDF system as approximately \qty[round-mode=places, round-precision=3]{-24896.350029}{\hartree}.

That some interaction graphs may not accurately represent the spatial topology of a molecule is a known issue in graph-based fragmentation methods~\cite{seeber2023}. The shape of 1KDF suggests that strong interactions should be expected between collections of particles that are close in space, but not close in the standard bond graph $G$.  The impacts of those interactions would then only appear in contribution terms $\tilde{V}_{\bm{u}}$ corresponding to quite large convex subgraphs $G\pdbrack{\bm{u}}$ which also include at least all those particles along a shortest path between the two originally considered. As mitigation, we employed instead an interaction graph $G$ with a distance-thresholded edge set $E = \Set{\pdbrace{i, j} \subseteq \pdbrack{M} \given i \neq j, \|R_i - R_j\| \leq r_{\text{cut}}}$. This is a common approach in fragmentation methods, either explicitly or implicitly; see and cf\@., e.g.,~\cite{ricard2020, zhang2021, seeber2023}, as well as many other examples in the reviews cited above. Here, we used a cutoff radius of $r_{\text{cut}} = \qty{2.5}{\angstrom}$, which should be sufficient to capture at least any hydrogen bonds~\cite{harris1999} in addition to all covalent bonds.

In the convex SUPANOVA decomposition~\eqref{eq:supanova}, we use notional subproblem potentials $V_{\bm{u}}$ such that an evaluated $\mathcal{L}\pdbrack{V_{\bm{u}}}$ is the HF/cc-pVTZ total energy of the subsystem indexed by $\bm{u}$, calculated using PySCF. Dangling covalent bonds from subsystems were treated with standard hydrogen link atoms, placed using covalent radii~\cite{cordero2008} consistent with, e.g.,~\cite[(19)]{richard2012}. Rather than a nuclear decomposition, we used a fragment decomposition. The choice of fragments was intended to avoid two practical issues: firstly, the question of how to terminate a dangling double bond, and second, the possibility that two link atoms would be placed in close spatial proximity after subsystem excision, thus potentially biasing the relevant $\tilde{V}_{\bm{u}}$; on the latter, cf\@., e.g.,~\cite{collins2006, seeber2023}. To this end, we refined an initial fragmentation $F = \pdbrace{\pdbrace{i}}_{i=1}^M$ of $\pdbrack{M}$ using a simple heuristic algorithm, the details of which are given in Appendix~\ref{sec:fragmentation_algorithm}. Technically, this means that the decomposition is not in terms of $\mathcal{M}_g\pdbrack{G}$, but rather $\mathcal{M}_g\pdbrack{G'}$, where $G' = G / F$ is the \emph{quotient graph}~\cite{george1978} $G' = \pdparen{F, E'}$ with edge set $E' = \Set{\pdbrace{F_{i'}, F_{j'}} \given \text{$i' \neq j'$, $\exists \pdbrace{i, j} \in E$ with $i \in F_{i'}, j \in F_{j'}$}}$ derived from the original edge set $E$ of $G$.

\begin{figure}[t]
  {\centering
    \includegraphics{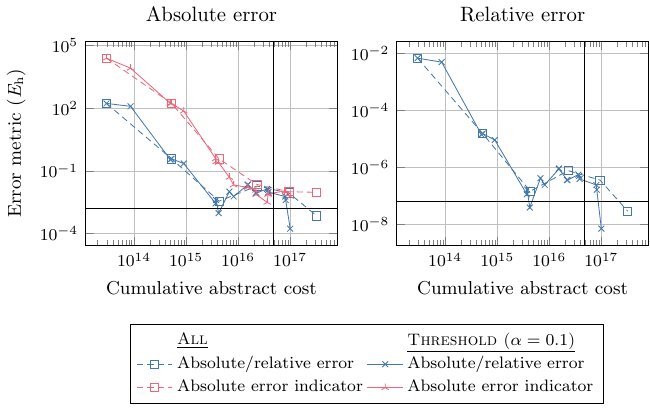}
    
  }
  \caption{Per-iteration error metrics for adaptive truncations of a convex SUPANOVA decomposition for the 1KDF antifreeze protein, obtained using \textsc{All} and \textsc{Threshold} selection strategies. The horizontal black line in each plot indicates chemical accuracy, \qty{1}{\kilo\cal\per\mole}. The vertical black line indicates the abstract cost of a full-system reference HF/cc-pVTZ calculation.}
  \label{fig:1kdf_results}
\end{figure}

Per-iteration results are plotted in Figure~\ref{fig:1kdf_results} for two adaptively-grown truncations of the convex SUPANOVA decomposition. One truncation was refined using the \textsc{All} selection strategy, and the other according to the \textsc{Threshold} strategy with $\alpha = \num{0.1}$.  Both calculations were terminated once the total, cumulative abstract cost expended passed the abstract cost of the reference calculation by a factor of two. The left-hand plot shows both the progression of the error indicators $\mathcal{E}_i$ at each iteration, as well as the true absolute error of the resulting value $S_i$ as measured against the reference total energy. The right-hand plot shows the per-iteration relative errors for the same calculations. The solid black vertical and horizontal lines on the plots indicate the abstract cost of the reference calculation, and the $\qty{1}{\kilo\cal\per\mole}$ threshold of chemical accuracy relative to reference, respectively. Since the first iteration of the algorithm produces in both cases only $S_0 = \mathcal{L}\pdbrack{V_{\emptyset}} = 0$ for zero abstract cost, these data points are not explicitly shown, nor will they be in similar plots to follow.

The absolute errors of the produced sums $S_i$ under both \textsc{All} and \textsc{Threshold} selection strategies decrease rapidly in early iterations, stabilizing around \qty{1e-02}{\hartree}. A further improvement is suggested shortly after the abstract costs of the iterative calculations pass that of the reference calculation. Excepting one early \textsc{Threshold} iteration, chemical accuracy relative to the reference value is only achieved in the final iteration. From a performance perspective, we would ideally hope to achieve chemical accuracy at a significantly reduced cost relative to the reference calculation. Although we do not achieve this, it should be kept in mind that, as plot (b) highlights, a consistent relative accuracy of at least six significant figures is achieved with approximately an order of magnitude speedup relative to reference. For larger systems again, such reliable relative accuracy is probably a more important consideration than true chemical accuracy, especially for total energy calculations.

As expected, there is no indication here of the inherent errors that would be expected in higher-order truncations of the standard BOSSANOVA expansion due to the presence of chordless cycles in $G'$. It is also positive to see that the error indicators $\mathcal{E}_i$ are generally reliable. If anything, they are a little conservative: the true error is sometimes significantly overestimated, but only rarely underestimated, and even then not by more than an order of magnitude. Further, although not plotted explicitly in Figure~\ref{fig:1kdf_results}, the propagated uncertainties $\text{d}S_i$ in these calculations hovered generally around $\qty{1e-06}{\hartree}$, and never exceeded $\qty{1e-05}{\hartree}$, assuming throughout a per-evaluation uncertainty of $\varepsilon = \qty{1e-08}{\hartree}$ consistent with involved convergence thresholds. This suggests that the unfavorable error propagation observed in~\cite{richard2014} for $n$-body truncations of the MBE is not in play here. It seems unlikely that this is due to any particular combinatorial property of the decomposition itself, but rather only that most terms $V_{\bm{u}}$ in the underlying MBE have zero combination coefficients; this would be consistent with, e.g.,~\cite{richard2014, lao2016}.

\begin{figure}[t]
  {\centering
    \includegraphics[width=0.5\textwidth]{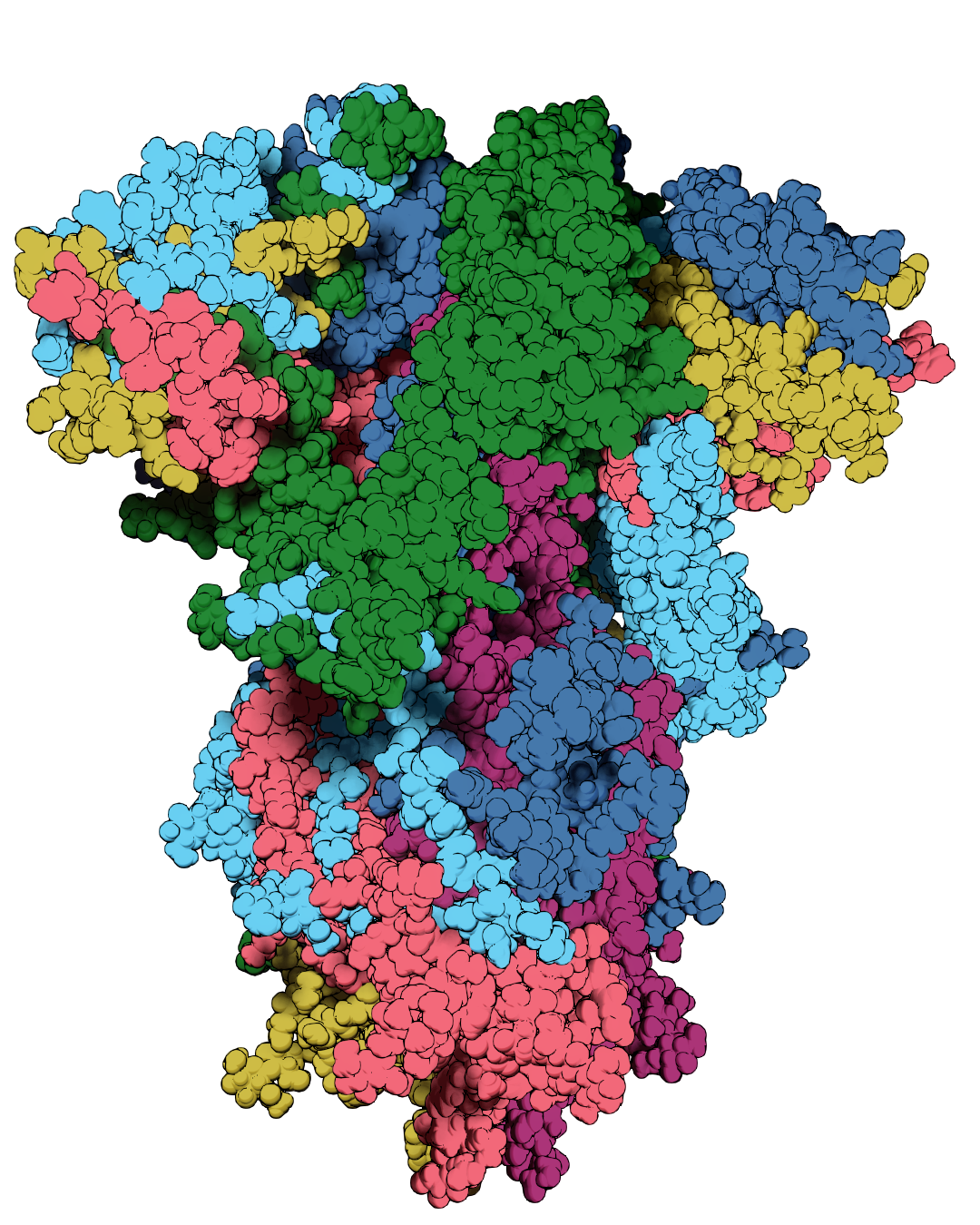}
    
  }
  \caption{Space-filling visualization of the SARS-CoV-2 spike glycoprotein~\cite[PDB key: 6VXX]{walls2020, walls2021}. Atoms are drawn as spheres of appropriate van der Waals radii~\cite{mantina2009}. Colors indicate membership of connected components of the full-system bond graph, chosen using a graph-coloring algorithm provided by NetworkX.}
  \label{fig:spike_glycoprotein_visualization}
\end{figure}

As a general feasibility test of the convex SUPANOVA approach, we performed a similar calculation on a much larger sample system, specifically, a model of the spike glycoprotein of the SARS-CoV-2 virus~\cite{walls2020}, again drawn from the Protein Data Bank~\cite[PDB key: 6VXX]{walls2021}. Preprocessing and hydrogenation with Open Babel produced here a system made up of 27 non-covalently bonded subsystems, each containing between \num{63} and \num{2690} non-hydrogen atoms, for a total of \num{46923} atoms. A space-filling visualization of the 6VXX system is given in Figure~\ref{fig:spike_glycoprotein_visualization}. We used the same methodology as above to choose a nuclear interaction graph $G$ ($r_{\text{cut}} = \qty{2.5}{\angstrom}$), a refined fragmentation $F$, and a final quotient graph $G' = G / F$. This $F$ included \num{7524} fragments, each containing up to \num{17} distinct atoms.

We performed a single adaptive calculation, using HF/cc-pVTZ subproblem potentials as above, and applying the \textsc{Threshold} strategy with $\alpha = 0.5$. The 6VXX system is of course far too large to allow a conventional reference calculation at any level of quantum theory. Indeed, we were not even able to evaluate an abstract cost for the complete system, since our abstract cost model involves the explicit calculation of some $\mathcal{O}\pdparen{N_{\text{AO}}^2}$ two-electron integrals; see Appendix~\ref{sec:abstract_cost_model}. Thus, we set no explicit termination criteria, and simply allowed the calculation to run for approximately \qty{46}{\hour}. At each iteration, subproblem potential evaluations were distributed in a hybrid fashion across up to \num{4864} cores of a compute cluster.

\begin{figure}[t]
  {\centering
    \includegraphics{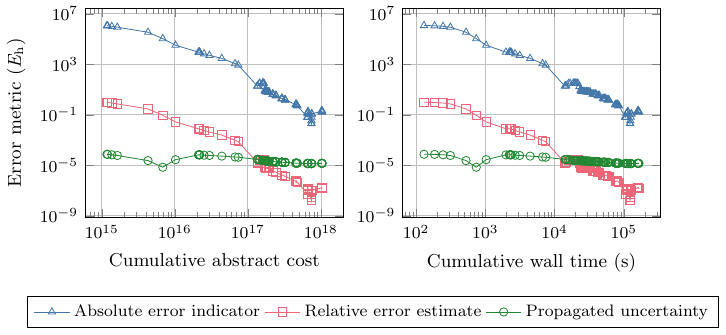}
    
  }
  \caption{Per-iteration error metrics for adaptive convex SUPANOVA approximation of the HF/cc-pVTZ total energy of the SARS-CoV-2 spike glycoprotein (PDB key: 6VXX)~\cite{walls2020, walls2021}.}
  \label{fig:spike_glycoprotein_results}
\end{figure}

Results for this calculation are shown in Figure~\ref{fig:spike_glycoprotein_results}. In the absence of a reference result, we can only consider the per-iteration error indicators. The left plot shows the absolute values of these as a function of cumulative abstract cost, also an estimate of the relative error calculated as $\abs{\mathcal{E}_i / S_i}$, and the propagated uncertainty of the calculation assuming per-subproblem thresholds of $\varepsilon = \qty{1e-08}{\hartree}$. The right-hand pane plots the same data, but measured as a function of the true cumulative wall-clock time required for the calculation. The similarity of the two plots supports the ``reasonableness'' of the abstract cost model we have used throughout, at least for HF. Some deviance here can be attributed to unavoidable parallel inefficiency.

The error indicators exhibit a clear and fairly smooth rate of decay. At termination, they suggest the calculation has approximated the total energy of the spike glycoprotein to again at least six significant figures of accuracy. Naturally, this conclusion depends entirely on the reliability of the error indicator. Although this is certainly suggested by results in the smaller 1KDF case, we do mention that we discussed cases in~\cite{barker2024} where the error indicator underperformed for reasons related to unrepresentative interaction graphs, and so further and more detailed validation would be in order before this could be relied upon. What can be said, however, is that the propagated uncertainty of this calculation remains basically flat throughout, and in particular, always stay several orders of magnitude less than the absolute value of the error indicator. So the numerics of the truncated sum are sound.

This test provides initial support to the idea that the convex SUPANOVA approach can be practically applied to extremely large molecules. It is to be stressed, however, that this is a preliminary calculation only, and further benchmarking of the underlying model is most certainly required. One potential target for such benchmarking is provided by an FMO calculation performed on the same protein~\cite{akisawa2021}. That calculation was much more methodologically sophisticated, and included in particular a high-quality treatment of correlation energy, so a direct comparison is not appropriate without further work. We do note that the computational resources applied in that study were substantially greater than ours, and it would be very interesting to investigate whether any computational advantage remained once all differentiating factors were properly taken into account.

\subsection{ML-SUPANOVA decompositions}\label{sec:ml_experiments}

When evaluating experimental results obtained by truncating a decomposition like~\eqref{eq:ml_supanova_expansion_explicit} of the true Born-Oppenheimer potential $V^{\text{BO}}$, we are impeded by the fact that there is no non-trivial molecular system for which a true reference solution can be obtained. We consider here only much smaller test systems than in the previous section. For these, we can obtain reasonable-quality approximate solutions, and these can function at least as rough references. Specifically, we investigate calculations on the linear alkane heptane (\ch{C7H16}), and the heterocyclic molecule limonin (\ch{C26H30O8}). Chemical structures for these systems were obtained from the ChemSpider database~\cite{csHeptane, csLimonin}, and then relaxed into plausible equilibrium conformations using a KS-DFT/B3LYP/cc-pVDZ geometry optimization under NWChem, using the \textsc{simint} ERI calculation library~\cite{pritchard2016}.

The poset grid under theoretical consideration is $\Pi = \mathcal{M}_g\pdbrack{G} \times \pdbrack{2N-1} \times \bbN$, as outlined above, but see the practical restriction discussed below. It is also possible to construct decompositions in terms of ``subgrids'' of this poset grid. These can be viewed as holding one or multiple of the parameters $\bm{u}$, $m$, and/or $p$ for potentials $V_{\pdparen{\bm{u}, m, p}}$ fixed, and doing so allows us to investigate the utility of the various axes. We consider the following subgrids of $\Pi$:

\begin{itemize}
\item $\Pi_{\text{non-ML}} = \mathcal{M}_g\pdbrack{G}$, with subproblem potentials fixed at a particular level of theory indexed by $\pdparen{m, p}$. This is a non-multilevel approximation. In cases where $\mathcal{M}_g\pdbrack{G} = \conn\pdbrack{G}$, this is a BOSSANOVA decomposition; where not, a more general convex SUPANOVA one.
\item $\Pi_{\text{GCM}} = \pdbrack{2N-1} \times \bbN$. All subproblem potentials are full-system potentials $V_{m,p}$, that is, with $\bm{u} = \pdbrack{M}$. Consistent with~\cite{barker2024}, where the idea was considered by extension of~\cite{zaspel2019}, we refer to this as a GCM (\emph{generalized composite method}) grid.
\item $\Pi_{\text{AbI}} = \mathcal{M}_g\pdbrack{G} \times \bbN$, with subproblem potentials fixed at a particular \emph{ab initio} method indexed by $m$, e.g., $m = 4$ for CCSD(T). As above, for cases where $\mathcal{M}_g\pdbrack{G} = \conn\pdbrack{G}$, this is an ML-BOSSANOVA decomposition.
\item $\Pi_{\text{BS}} = \mathcal{M}_g\pdbrack{G} \times \pdbrack{2N-1}$, with subproblem potentials using a fixed basis set indexed by $p$, e.g., $p = 2$ for cc-pCVTZ.
\end{itemize}

For each test case, we calculated a reference total energy using a single-point calculation via PySCF, at as high a level of theory as practically possible. We also used an alternative mechanism to derive a reference result, specifically, an implementation of the G4(MP2) method~\cite{curtiss2007a}, written in Python with some reference to~\cite{ernst2016}. We omitted here the initial DFT-based geometry optimization technically required by G4(MP2). In general, the component calculations required were also performed using PySCF, using monatomic energies generally calculated using MRCC, as described above.

Since G4(MP2) makes the frozen-core approximation, it is inherently unsuited to the calculation of accurate total energies, but might be reliable to around chemical accuracy~\cite{karton2017, das2021} for the calculation of energy differences such as atomization energies. An atomization energy potential can be constructed from the Born-Oppenheimer potential as
\begin{equation}
  V^{\text{atom}}\pdparen{X_1, \ldots, X_M}
  \coloneq \pdbrack*{\sum_{A=1}^M V^{\text{BO}}\pdparen{X_A}} - V^{\text{BO}}\pdparen{X_1, \ldots, X_M},
\end{equation}
cf., e.g.,~\cite{helgaker2000}. Similar potentials can also be derived from approximations to the Born-Oppenheimer potential like any of the $V_{\pdparen{\bm{u}, m, p}}$ terms implicitly involved in~\eqref{eq:ml_supanova_expansion_explicit}, taking care to match the level of theory $\pdparen{m, p}$ for the monatomic calculations. Since monatomic energies $V_{m,p}\pdparen{X_A}$ can be precalculated and reused, we assume that the abstract costs of such potentials are precisely the same as the standard total-energy versions.

For practical reasons, we restricted the poset axes $\pdbrack{2N-1}$ and $\bbN$ for \emph{ab initio} theory and basis set. For atomization energy calculations, we restricted the \emph{ab initio} theory axis to contain four elements corresponding to HF, MP2, CCSD, and CCSD(T), and the basis set axis to contain seven elements, corresponding to cc-pCV$n$Z for $n = 2$ (cc-pCVDZ) through 8 (cc-pCV8Z). For total energy calculations, we further restricted these axes according to the level of theory used to derive the reference result. While we would anticipate larger axes to produce more accurate approximations of $V^{\text{BO}}$, any additional accuracy would confound a comparison against the reference result. For atomization energy calculations, we do not consider the subgrids $\Pi_{\text{non-ML}}$ or $\Pi_{\text{BS}}$; $\Pi_{\text{AbI}}$ is, however, considered, fixed for CCSD(T).

\begin{figure}[t]
  {\centering
    \includegraphics{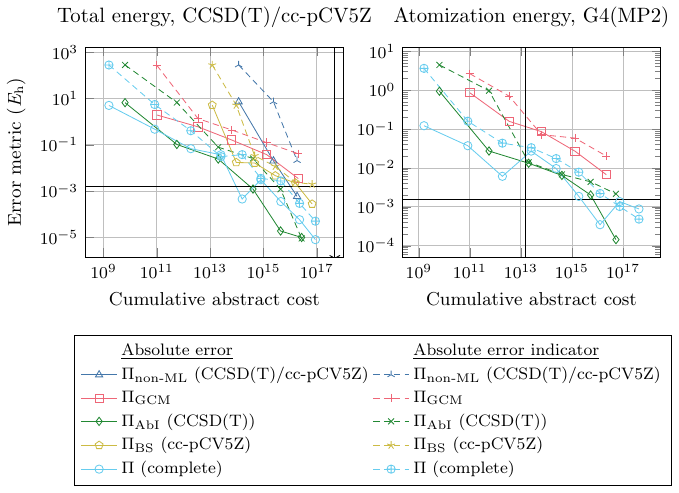}
    
  }
  \caption{Per-iteration error metrics for adaptive truncations of the ML-SUPANOVA decomposition for heptane, and for subgrid decompositions.}
  \label{fig:heptane_results}
\end{figure}

The reference all-electron CCSD(T)/cc-pCV5Z total energy of heptane was approximately \qty[round-mode=places, round-precision=3]{-276.369764}{\hartree}, and the atomization energy of heptane as per G4(MP2) was approximately \qty[round-mode=places, round-precision=3]{3.480747}{\hartree}. We used here only the standard covalent bond graph $G'$. The corresponding fragmentation $F$ of the nuclear indices, with each fragment $F_i$ containing the indices of exactly one backbone carbon atom and also those of any hydrogen atoms bonded to it, is both as delivered by the heuristic algorithm, and precisely consistent with that suggested for the original BOSSANOVA decomposition. We applied here only the \textsc{All} selection strategy.\footnote{Note that for the \textsc{Threshold} and \textsc{Best} strategies, some minor adjustment of Algorithm~\ref{alg:adaptive} would be required to handle the fact that, e.g., $\tilde{V}_{\pdparen{\emptyset,m,p}} = 0$.} Plots of per-iteration error metrics for adaptive truncations of the complete poset grid $\Pi$ and of the variously-considered subgrids listed above are given in Figure~\ref{fig:heptane_results}. The left-hand plot shows accuracy results for total energy calculations, as measured against the CCSD(T)/cc-pCV5Z reference, and corresponding error indicators. The right-hand plot is equivalent, but for atomization energy calculations, measured against the G4(MP2) reference result. We correct in this plot an error in~\cite[Fig.~7.1]{barker2024}.\footnote{In~\cite{barker2024}, the absolute error values plotted were measured relative, not to the G4(MP2) atomization energy of heptane as stated, but rather the ccCA-PS3~\cite{deyonker2009} atomization energy of heptane, see~\cite[Tab.~7.2]{barker2024}. The error behaviour shown in the original~\cite[Fig.~7.1]{barker2024} seems, in this light, consistent with the reasoning that led to the choice of the G4(MP2) value as reference, rather than the ccCA-PS3 atomization energy.}

For the total energy results, we see, in short, that all considered sequences of truncations lead (with the narrow exception of the GCM result) to chemically-accurate approximations of the reference result, and at speedups of at least an order of magnitude. It is interesting to observe that the generally best performance is delivered by the subgrid fixed at CCSD(T); since $G$ is a simple chain, here $\mathcal{M}_g\pdbrack{G} = \conn\pdbrack{G}$, and this is precisely an ML-BOSSANOVA expansion. The performance here is any case comparable with truncations for the complete poset grid $\Pi$, while the fixed-basis $\Pi_{\text{BS}}$ grid performs less well. The error indicators for all calculations are generally well-behaved, although sometimes prone to overestimation of the true error, particularly for the $\Pi_{\text{non-ML}}$ and $\Pi_{\text{AbI}}$ cases.

For the calculation of atomization energies, no cost benefit is obtained relative to the reference calculation. We do not read much into this: heptane is still only a small system, and it is unsurprising that the CCSD and CCSD(T) calculations involved in the ML-SUPANOVA truncation are more expensive than the up-to-MP2 calculations required for G4(MP2). Again, the per-iteration absolute errors for the three tested grids eventually trend progressively downwards, although the results for the complete grid $\Pi$ seem to spend considerable time in a pre-asymptotic regime. Truncations for both $\Pi$ and $\Pi_{\text{AbI}}$ consistently outperform those for $\Pi_{\text{GCM}}$ at similar cost, and both come to approximate the G4(MP2) atomization energy to within chemical accuracy. A slight advantage is noticeable here for the complete grid $\Pi$, and the final iterations for that grid suggest a slightly faster rate of ``convergence'' relative to those for $\Pi_{\text{AbI}}$. The absolute error indicators for all grids are generally reliable, tracking the true absolute error closely. When considering the  absolute error and the absolute error indicator for $\Pi$ in the final iterations, we should remember that the G4(MP2) result may itself not be reliable much past the chemical accuracy cutoff. It remains therefore possible that the error indicator is valid and the true FCI/CBS atomization energy of heptane is indeed being approximated more and more closely. This cannot be the case for $\Pi_{\text{AbI}}$, which can at best approximate the CCSD(T)/CBS atomization energy.

\begin{figure}[t]
  {\centering
    \includegraphics[width=0.5\textwidth]{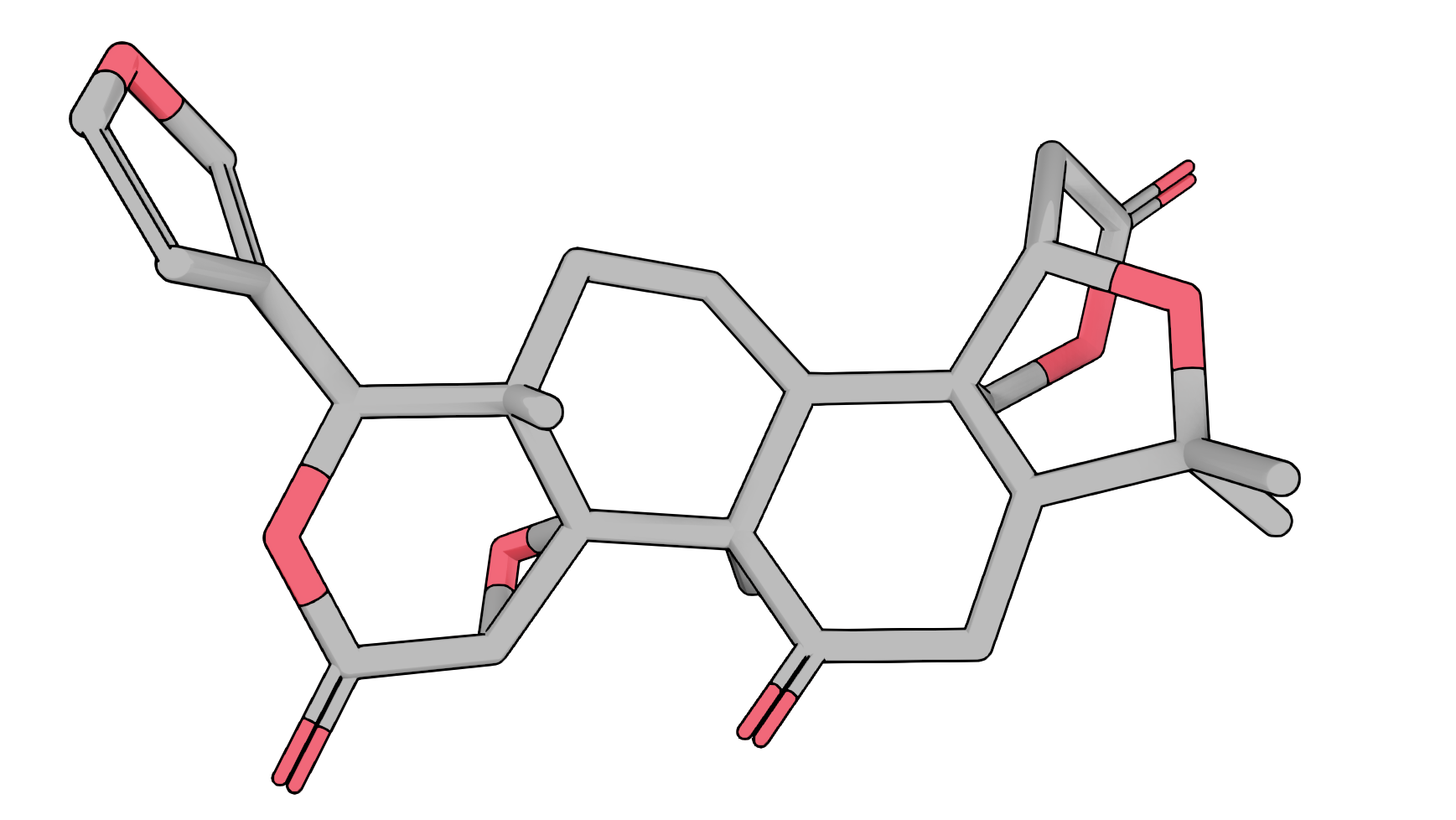}
    
  }
  \caption{Stick-model visualization of limonin (\ch{C26H30O8}), after KS-DFT/B3LYP/cc-pVDZ geometry optimization. Sticks indicate single and double covalent bonds between non-hydrogen atoms; hydrogen atoms themselves, and bonds to them, are not shown.}
  \label{fig:limonin_visualization}
\end{figure}

The second system we consider, limonin, is of interest particularly due to the highly cyclic nature of its covalent bond graph $G'$, clearly visible in the visualization in Figure~\ref{fig:limonin_visualization}. Cyclic structures remain after fragmentation with the heuristic algorithm: the resulting quotient graph $G' = G / F$ contains five chordless cycles of length at least five. This indicates that the standard ML-BOSSANOVA decomposition using $\conn\pdbrack{G}$ cannot succeed, and the use of $\mathcal{M}_g\pdbrack{G}$ is more appropriate; although we do not give details here for reasons of space, this was established for the single-level BOSSANOVA case in~\cite[Sec.~6.6]{barker2024}. It was infeasible to calculate a reference total energy for limonin using CCSD(T) at an appropriately high basis set. Instead, we consider a reference total energy for limonin calculated at the MP2/cc-pCVQZ level of theory: approximately \qty[round-mode=places, round-precision=3]{-1609.907486}{\hartree}. The atomization energy of limonin as per G4(MP2) was calculated as approximately \qty[round-mode=places, round-precision=3]{11.183448}{\hartree}.

\begin{figure}[t]
  {\centering
    \includegraphics{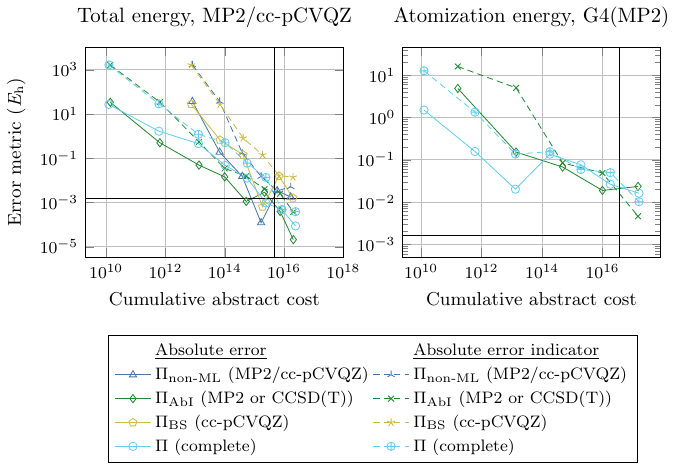}
    
  }
  \caption{Per-iteration error metrics for adaptive truncations of~\eqref{eq:ml_supanova_expansion_explicit} for limonin, as well as subgrid decompositions.}
  \label{fig:limonin_results}
\end{figure}

Per-iteration results for \textsc{All}-strategy adaptive calculations for limonin are shown in Figure~\ref{fig:limonin_results}, with a layout equivalent to those for heptane. We omit from consideration also the subgrid $\Pi_{\text{GCM}}$, since the results available up to practicality are not usefully comparable. 
We begin with approximations of the MP2/cc-pCVQZ total energy. In their final iterations, all such calculations, either approach ($\Pi_{\text{non-ML}}$ and $\Pi_{\text{BS}}$) or comfortably exceed ($\Pi_{\text{AbI}}$ and $\Pi$) chemical accuracy with respect to the reference result. Their respective error indicators are still sometimes prone to overestimation of the error, even late in the calculation. There does not seem to be either a clear speedup relative to the reference, or a clear winner amongst the subgrids trialled: the $\Pi_{\text{AbI}}$ subgrid is the first to reach chemical accuracy, but the errors produced by truncations for $\Pi$ seem to decay slightly more regularly.

For approximations of atomization energies, it is again interesting to see that while the error indicators are broadly reliable, they do underestimate the true error in the very final iterations. It is difficult to be conclusive here. As regards accuracy, neither calculation produces an agreement with the reference value better than \qty{0.01}{\hartree}. 
Here, it may be the case that further iterations would produce a closer agreement, but again, since we would expect the truncations to approach the true $V^{\text{atom}}$ --- or, given the restricted grid, the approximation of $V^{\text{atom}}$ for CCSD(T)/cc-pCV8Z --- this would again depend on the accuracy of the G4(MP2) energy itself.

Although we are limited in our ability to draw conclusions by the quality of the available reference results, it does still seem clear that ML-SUPANOVA truncations are systematically improvable. It is important to note that the highly cyclic structure of the limonin system is not an issue. For atomization-energy calculations, there is still no clear performance benefit here compared to a simple application of G4(MP2). But the G4(MP2) cost is here significantly higher than that for heptane, and there must quickly come a point where even G4(MP2) will become prohibitively expensive. It is here that we still anticipate that the ML-SUPANOVA approach will truly come into its own.

\section{Concluding Remarks}\label{sec:conclusion}

In this article, we investigated the working equations of single- and multilevel energy-based fragmentation methods. The technique of Möbius inversion, well-known to the cluster expansion methods community and easily extensible from work there, led to the unifying expansion form~\eqref{eq:poset_grid_multilevel}. As well as reproducing the energy equations of a number of existing multilevel fragmentation methods, this expansion provides for a further extension again, namely the efficient approximation of individual many-body terms in the broad style of a composite method, using a framework based on the combination technique.

The algorithm we have given here is clearly capable of refining ML-SUPANOVA truncations in a way that produces a systematic increase in accuracy. However, some questions do remain regarding their function as a mechanism for computational speedup. We can reasonably anticipate that application of ML-SUPANOVA is likely to pay dividends for much larger systems than those we have considered here. A further investigation of ``true'' adaptivity, with, e.g., the \textsc{Threshold} strategy, would be appropriate. But the primary issue will remain one of validation, given the inherent unavailability of true reference solutions. Two complementary approaches suggest themselves. The first would be an effort to develop rigorous bounds on the decay properties of the MBE subproblem potentials $\tilde{V}_{\bm{u}}$, and then also on their more general multilevel equivalents. The second would be detailed practical comparison against experimental results. For this, it would be necessary to obtain from truncated ML-SUPANOVA expansions also point evaluations of the nuclear gradient, rather than just of total and/or atomization energies. This adjustment, very well-known in existing fragmentation methods, requires in our setup just a different choice of linear operator.

Were it not already clear, let us stress that the combinatorial tools we have applied here are textbook knowledge. Nevertheless, they allowed us to identify and avoid a subtle but important issue in the existing (ML-)BOSSANOVA framework. It seems likely that careful investigation of more recent developments in combinatorics, order theory, and the theory of graph convexities may lead to helpful further insights. It would be particularly interesting to explore whether and to what extent it might be possible to derive a general scheme for the construction of a interaction graph $G$ that both captures the topological structure of a molecular system while also ensuring that $\mathcal{M}_g\pdbrack{G}$ is truly a convex geometry, but this may prove a difficult task.

As mentioned above, we remain particularly interested in the use of subproblem potentials employing quantum embedding, which are explicitly catered for by our formal setup. For example, if one were to build a set of potentials that applied either a WFT-in-WFT or WFT-in-DFT quantum embedding scheme, it would be natural to use the \emph{ab initio} theory index $m$ and basis set index $p$ to specify the treatment of the embedding region. There are other possibilities; one might also consider a multilevel treatment of the environment region. Here, there will be practical issues, since the cost of evaluating each $V_{\pdparen{\bm{u}, m, p}}$ will presumably scale superlinearly in the size of the full system.

\section*{Acknowledgements}

The authors gratefully acknowledge support from the Hausdorff Center for Mathematics (DFG project number 390685813) and the Collaborative Research Centre 1060 (DFG project number 211504053). The calculations reported were performed using the high-performance computing resources of Fraunhofer SCAI.
We thank co-workers at the Institute for Numerical Simulation and at Fraunhofer SCAI, particularly Astrid Maa{\ss}, for many fruitful discussions. David Feller provided a helpful discussion regarding some of the basis sets used here.

\appendix
%\section*{Appendix}
\section{Proofs}\label{sec:proofs}

In this section, we provide proofs for some of the results given in the text. With the exception of the first proof (for Prop.~\ref{prop:meets_of_generating_antichain}), these are drawn almost verbatim from~\cite{barker2024}, with some small adjustments for clarity and also consistency within this article. We use some basic terminology and concepts not explicitly introduced above; see, e.g.,~\cite{rota1964, aigner1997, stanley2012}.

The first proof uses just the same standard results and idea of proof as on~\cite[p.~7481]{lafuente2005}.

\begin{proof}[Proof of Proposition~\ref{prop:meets_of_generating_antichain}]
  If $s \notin I$, then immediately $D^{\pdparen{I}}_s = 0$, so let $s \in I$. Define $J \coloneq I \cup \pdbrace{\hat{1}_J}$, such that $\hat{1}_{J} >_J t$ for any $t \in I$. By Lemma~\ref{lem:equivalent_combination_coefficient}, $D^{\pdparen{I}}_s = -\mu_J\pdparen{s, \hat{1}}$. Note that the interval $\pdbrack{s, \hat{1}}_J$ is itself a meet semilattice, so a lattice by~\cite[Prop.~3.3.1]{stanley2012}. If $\mu_J\pdparen{s, \hat{1}} = -D^{\pdparen{I}}_s \neq 0$, then $s$ must be a meet of coatoms of $[s, \hat{1}]_J$ by~\cite[Cor.~3.9.5]{stanley2012}, and these coatoms are in $A$ by construction.
\end{proof}

\begin{proof}[Proof of Lemma~\ref{lem:unique_consistent_order_ideal}]
  Let $P$, $Q$, $I'$, and $A'$ be as in the statement of the claim.  Further, let $I$ be some finite order ideal of $P$ which is combination-consistent with $I'$. Since $I$ is finite, it is generated by an antichain $A \subseteq I$; we want to show that $A = A'$.

  The proof is by contradiction. Suppose that $A \neq A'$. Then there exists some $a \in A$ such that $a \notin A'$, or otherwise, some $a' \in A'$ such that $a' \not\in A$. Begin by supposing the former.  Since $a$ is maximal in $I$, we have $D^{\pdparen{I}}_a = \mu_P\pdparen{a, a} = 1$ by definition. If $a \not\in Q$, then by~\eqref{eq:combination_consistency}, also $D^{\pdparen{I}}_a = 0$, a contradiction. So $a \in Q$. It must be that $a \in I'$, for otherwise, $\hat{D}^{\pdparen{I'}}_a = 0$, again a contradiction. Since $a \not\in A'$, there exists some $a^\dagger \in A'$ such that $a^\dagger > a$. Just as above, $\hat{D}^{\pdparen{I'}}_{a^\dagger} = 1$, but since $a$ is maximal in $I$, we know that $a^\dagger$ cannot also be an element of $I$, so also $D^{\pdparen{I}}_{a^\dagger} = 0$, contradicting~\eqref{eq:combination_consistency}.

  It must then be that there exists $a' \in A'$ such that $a' \not\in A$. Just as above, $\hat{D}^{\pdparen{I'}}_{a'} = 1$. If $a' \not\in I$, then $D^{\pdparen{I}}_{a'} = 0$, contradicting~\eqref{eq:combination_consistency} since $a' \in A' \subseteq Q$. So $a' \in I$. Then there exists some $a^\dagger \in A \subseteq P$ such that $a^\dagger > a'$, and $D^{\pdparen{I}}_{a^\dagger} = 1$. Since this is nonzero, it must be that $a^\dagger \in Q$, by~\eqref{eq:combination_consistency}. But since $\hat{D}^{\pdparen{I'}}_{a^\dagger}$ is nonzero only when $a^\dagger \in I'$, we have that $a'$ is not maximal in $I'$, so $a' \notin A'$, a contradiction.
\end{proof}

In the following proof, we make use again of the same ideas as on~\cite[p.~7481]{lafuente2005}, particularly for direction $\pdparen{\Leftarrow}$. We correct here a minor omission in the proof given in~\cite{barker2024}, namely, a trivial case when $s \notin I$, and also correct a minor error related to the use of the Crosscut Theorem.

\begin{proof}[Proof of Theorem~\ref{thm:combination_consistency}]
  $\pdparen{\Leftarrow}$ Let $P$ be a locally finite meet semilattice with a $\hat{0}$, and let $Q \subseteq P$ be a meet subsemilattice of $P$. Let $I'$ be an arbitrary finite order ideal of $Q$ with generating antichain $A'$, and let $I = \pdprod{A'}_P$ be the finite order ideal of $P$ generated by $A'$ in $P$.

  Fix some $s \in P$. If $s \notin I$, then clearly $D^{\pdparen{I}}_{s} = 0$, and if also $s \in Q$, similarly $\hat{D}^{\pdparen{I'}}_s = 0$. So assume $s \in I$. Define $J \coloneq I \cup \pdbrace{\hat{1}_J}$ and $J' \coloneq I' \cup \pdbrace{\hat{1}_{J'}}$, such that $\hat{1}_{J} >_J t$ for any $t \in I$, and $\hat{1}_{J'} >_{J'} t'$ for any $t' \in I'$. If $s$ is a meet of coatoms of $J$, then also $s \in J'$, since each such coatom is in $A'$ by construction and $Q$ is a meet subsemilattice of $P$. Further, noting that the coatoms of $\pdbrack{s, \hat{1}_J}_J$ and $\pdbrack{s, \hat{1}_{J'}}_{J'}$ must be the same, it follows that $D^{\pdparen{I}}_s = -\mu_J\pdparen{s, \hat{1}_J} = -\mu_{J'}\pdparen{s, \hat{1}_{J'}} = \hat{D}^{\pdparen{I'}}_s$ by Lemma~\ref{lem:equivalent_combination_coefficient} and the Crosscut Theorem~\cite[Cor.~3.9.4]{stanley2012}, using those sets of coatoms as the subset $X$ in the statement of the latter result. If $s$ is not a meet of coatoms of $J$, then $D^{\pdparen{I}}_s = 0$ by~\cite[Cor.~3.9.5]{stanley2012}; if here also $s \in Q$, then similarly $\hat{D}^{\pdparen{I'}}_s = 0$, and we are done.

  For $\pdparen{\Rightarrow}$, the proof is by contradiction. Let $P$ be a locally finite meet semilattice with a $\hat{0}$, and let $Q$ be a combination-consistent subposet of $P$.  Suppose that $Q$ is, however, not a meet subsemilattice of $P$. Then there exist (at least) two distinct elements $t, t' \in Q$ such that $t \wedge_P t' \not\in Q$. Let $I' = \pdprod{\pdbrace{t, t'}}_Q$ be the finite order ideal of $Q$ generated by those two elements. Since $Q$ is combination-consistent with $P$, there exists a finite order ideal $I$ of $P$ that is combination-consistent with $I'$. By Lemma~\ref{lem:unique_consistent_order_ideal}, this $I = \pdprod{\pdbrace{t, t'}}_P$. Form as above $J \coloneq I \cup \pdbrace{\hat{1}_J}$. Since the only subset of $\pdbrace{t, t'}$ whose meet is $t \wedge_P t'$ is $\pdbrace{t, t'}$ itself, it follows from the Crosscut Theorem~\cite[Cor.~3.9.4]{stanley2012} that $D^{\pdparen{I}}_{t \wedge_P t'} = -\mu_J\pdparen{t \wedge_P t', \hat{1}_J} = 1$. But, since $I$ is combination-consistent with $I'$, and since $t \wedge_P t' \not\in Q$, also $D^{\pdparen{I}}_{t \wedge_P t'} = 0$, a contradiction.
\end{proof}

\section{Abstract Cost Model}\label{sec:abstract_cost_model}

In this section, we outline the abstract cost model used in the calculations described above. The model is, in essence, just a collection of asymptotic expressions for the costs of individual algorithms, see again Section~\ref{sec:fundamentals} and, e.g.,~\cite{szabo1996, cances2003, echenique2007, helgaker2000, jensen2017}, with the inclusion of some tunable constant factors. We simply state the model here, and refer the interested reader to~\cite{barker2024} for more details as well as some additional supporting references.

The computational cost required to evaluate the total energy for a molecular system at some level of \emph{ab initio} theory is ultimately a function of the specific molecular geometry as well as the selected basis set. The number and species of atoms determines the number of basis functions $N_{\text{AO}}$. All of the quantum-chemistry algorithms discussed here involve the calculation of formally $\mathcal{O}\pdparen{N_{\text{AO}}^4}$ two-electron integrals (ERIs) over atomic-orbital basis functions, but very many of these will usually be negligibly small~\cite{gill1994, echenique2007}. Since ERI negligibility is most commonly an expression of spatial separation between particles, this effect will be less pronounced for a more compact molecular system such as an individual fragment. In the following, we assume knowledge or at least an estimate of the number $N_{\text{ERI}}$ of non-negligible ERIs. We also include a tuneable factor $f_{\text{ERI}}$ that can be used to weight the cost of their calculation as $f_{\text{ERI}}N_{\text{ERI}}$.

We assume the use of a direct SCF algorithm~\cite{almloef1982} for Hartree-Fock calculations. The work at each iteration involves first calculating all distinct ERIs, and then solving the Roothaan-Hall equations. Assuming that some $N^{\text{HF}}_{\text{iter}}$ iterations are required on average, then, the cost of a Hartree-Fock calculation is modelled as
\begin{equation*}
  \mathcal{C}^{\text{HF}}\pdparen{\cdots}
  =
  N^{\text{HF}}_{\text{iter}}
  \pdparen{
    f_{\text{ERI}}N_{\text{ERI}}
    +
    N_{\text{AO}}^3
  },
\end{equation*}
cf\@. the cost model in~\cite{chinnamsetty2018}. We hide the various arguments and parameters of the cost function inside an ellipsis for notational simplicity.

It is common to see the cost of an MP2 calculation stated as $\mathcal{O}\pdparen{N_{\text{AO}}^5}$, but this can be expressed a little more tightly. One necessarily begins with a Hartree-Fock calculation, which provides the first term in the following formula. The remaining calculation depends on the number of virtual orbitals $N_{\text{virt}}$, and also on the number of correlated orbitals $N_{\text{corr}}$, which is just the number of occupied orbitals $N_{\text{occ}}$ for an all-electron calculation but somewhat lower for a calculation made under the frozen-core approximation. The actual evaluation of the MP2 contribution term requires only $\mathcal{O}\pdparen{N^2_{\text{corr}}N^2_{\text{virt}}}$ operations, cf.~\cite{szabo1996}, giving the second term below. However, these involve ERIs over molecular orbitals, which must be computed from the atomic-orbital ERIs. After these latter are explicitly calculated (third term), four tensor contractions (remaining terms) deliver the final tensor of molecular ERIs; for details on this transformation, see, e.g.,~\cite{wong1996}. Ordering these optimally under the assumption that $N_{\text{corr}} < N_{\text{virt}}$, then,
\begin{multline*}
  \mathcal{C}^{\text{MP2}}\pdparen{\cdots}
  =
  \mathcal{C}^{\text{HF}}\pdparen{\cdots}
  + N^2_{\text{corr}}N^2_{\text{virt}}
  + f_{\text{ERI}}N_{\text{ERI}}\\
  + N_{\text{corr}}N_{\text{ERI}}
  + N^2_{\text{corr}}N_{\text{AO}}^3
  + N^2_{\text{corr}}N_{\text{virt}}N_{\text{AO}}
  + N^2_{\text{corr}}N^2_{\text{virt}}N_{\text{AO}}.
\end{multline*}
We model the costs of a general $n$th-order coupled cluster calculation more simply, as
\begin{equation*}
  \mathcal{C}^{\text{CC(n)}}\pdparen{\cdots}
  =
  f_{\text{ERI}}N_{\text{ERI}}
  + N_{\text{AO}}\pdparen{N_{\text{corr}} + N_{\text{virt}}}^4
  + N^{\text{CC}}_{\text{iter}}N^n_{\text{corr}}N^{n+2}_{\text{virt}};
\end{equation*}
that is, as the cost of calculating non-negligible atomic ERIs (first term), transforming them into molecular ERIs (second term, modelled without as much detail as for MP2), and then iteratively solving the necessary amplitude equations, which is assumed to require $N^{\text{CC}}_{\text{iter}}$ iterations on average; for asymptotic costs of coupled cluster calculations, see again, e.g.,~\cite{kallay2001, kallay2005, bartlett2007, karton2022}. This assumes that the cost difference of the frozen-core approximation can be captured by replacing $N_{\text{occ}}$ by $N_{\text{corr}}$. Finally, by extension, the cost of evaluating such a calculation but also with a perturbative approximation for the effects of subsequent-order excitations is modelled as
\begin{equation*}
  \mathcal{C}^{\text{CC(n)(n+1)}}\pdparen{\cdots}
  =
  \mathcal{C}^{\text{CC(n)}}\pdparen{\cdots}
  +
  N^{n+1}_{\text{corr}}N^{n+2}_{\text{virt}}.
\end{equation*}

To obtain the abstract costs used in this article, we used the factor values $N^{\text{HF}}_{\text{iter}} = 15$, $N^{\text{CC}}_{\text{iter}} = 15$, and $f_{\text{ERI}} = 50$. We used PySCF to derive values of $N_{\text{AO}}$, $N_{\text{corr}}$, and $N_{\text{virt}}$, and estimated $N_{\text{ERI}}$ as per the standard Cauchy-Schwarz bound~\cite{haeser1989}, with the ERIs required to evaluate the bound calculated using \textsc{libcint}~\cite{sun2015}. ERIs were deemed negligible below a bound of \qty{1e-12}, evaluated per basis shell, and counted up to permutational symmetry.

\section{Heuristic Fragmentation Algorithm}\label{sec:fragmentation_algorithm}

In this section, we describe informally the heuristic algorithm used to derive the fragmentations used in Section~\ref{sec:experiments}. Particularly similar heuristics are used in~\cite{collins2012, le2012}. As an initial candidate fragmentation, we take $F = \pdbrace{\pdbrace{i}}_{i=1}^M$, so assigning the index of each atom in the system to its own fragment. Fragments are then repeatedly coalesced in two separate phases.

The first phase collects any bonded hydrogen atoms with their ``parents'', and removes the possibility that calculating any subproblem potential $V_{F_{\bm{u}}}$ would require severing a double- or higher-order bond. Iterating over pairs, two fragments $F_i$ and $F_j$ are combined if either a) a non-single covalent bond exists between an atom in each fragment, or b) a hydrogen atom in one fragment is bonded to an atom in the other fragment. This is repeated until a steady state is obtained.

The second phase is intended to avoid the possibility whereby two covalent bonds leading to the same external atom must be severed in the calculation of some $V_{F_{\bm{u}}}$, which would result in the introduction of two link atoms in very close proximity to one another. We iterate again over pairs of fragments, looking for two particular cases. In the first case, a single atom in one fragment $F_i$ is bonded to two atoms in another fragment $F_j$, where the two fragments are adjacent in the quotient graph. In the second case, atoms in two fragments $F_i$ and $F_j$ are bonded, and either of those two atoms is also bonded to an atom in a third fragment $F_k$ that is itself adjacent to both $F_i$ and $F_j$. Again, the second phase terminates when a steady state is encountered.

Note that the output of this algorithm is dependent upon the ordering of the fragments, and on how fragment pairs are enumerated at each iteration of each phase.

\printbibliography

%\bibliographystyle{amsplain}
%\bibliography{mlebfm}

\end{document}